\RequirePackage{fix-cm}
\documentclass[smallextended]{svjour3}       
\smartqed  
\usepackage{xspace}
\usepackage{algorithm}
\usepackage{subfigure}
\usepackage{amsfonts}
\usepackage{amsmath}
\usepackage{enumerate}
\usepackage{framed}
\usepackage{graphicx}
\usepackage{lscape}
\usepackage{bm}
\usepackage{color}
\usepackage{multicol}
\usepackage{cite}

\newcommand{\propref}[1]{{Proposition~\ref{#1}}}

\renewcommand{\eqref}[1]{{(\ref{#1})}}

\newcommand{\R}{\mathbb{R}}
\newcommand{\N}{\mathbb{N}}

\newtheorem{dfn}{Definition}
\newtheorem{thm}{Theorem}
\newtheorem{lem}{Lemma}
\newcommand{\lemref}[1]{{Lemma~\ref{#1}}}
\newtheorem{ass}{Assumption}
\newtheorem{prp}{Proposition}
\newtheorem{rem}{Remark}

\usepackage{graphicx}
\begin{document}
	
	\title{Strong convergence rates  of  an exponential integrator  and finite elements method for time-fractional SPDEs driven by Gaussian and non-Gaussian noises}
	
	
	\titlerunning{Strong convergence of an exponential integrator for time-fractional SPDEs}        
	
	\author{Aurelien Junior Noupelah, Antoine Tambue.
	}
	
	\authorrunning{A. J.  Noupelah, A. Tambue} 
	
	\institute{A. J.  Noupelah \at
		Department of Mathematics and Computer Sciences, University of Dschang, P.O. BOX 67, Dschang, Cameroon\\
		\email{noupsjunior@yahoo.fr}   \\
		\and
		A. Tambue (Corresponding author) \at
		 Department of Computer science, Electrical engineering and Mathematical sciences,  Western Norway University of Applied Sciences, Inndalsveien 28, 5063 Bergen.\\  
		\email{antonio@aims.ac.za, tambuea@gmail.com} 
	}
	\date{Received: date / Accepted: date}

	\maketitle
	
	\begin{abstract}
		In this work, we provide the first strong convergence result of numerical approximation of a general second order semilinear stochastic fractional order evolution equation involving a Caputo derivative in time of order $\alpha\in(\frac 34, 1)$ and driven by Gaussian and non-Gaussian noises simultaneously more useful in concrete applications. The Gaussian noise considered here is a Hilbert space valued Q-Wiener process and the non-Gaussian noise is defined through compensated Poisson random measure associated to a L\'evy process. The linear operator is not necessary self-adjoint. The  fractional stochastic partial differential equation is discretized in space by the finite element method and in time by a variant of the exponential integrator  scheme. We investigate the mean square error estimate of our fully discrete scheme and the result shows how the convergence orders depend on the regularity of the initial data and the power of the fractional derivative.
		
		
	\end{abstract}

	\section{Introduction}
	\label{intro}
	We consider the  following SPDE  with initial value 
	\begin{equation}
		\label{model_timfrac}
		\left\{
		\begin{array}{ll}
			\partial^{\alpha}_tX(t)=AX(t)+F(X(t))+B(X(t))\frac{dW(t)}{dt}+\frac{\int_{\mathcal{X}}G(z,X(t))\widetilde{N}(dz,dt)}{dt}, \\
			X(0)=X_0,\,\,\,\,t\in[0,T],
		\end{array}
		\right.
	\end{equation} 
	on the Hilbert space $H=\left(L^2(\Lambda),\left\langle \cdot,\cdot\right\rangle_H,\left\|\cdot\right\|\right)$,   $\Lambda\subset\R^d$, $d=1,2,3$, where $T>0$ is the final time, $A$ is a linear operator which is unbounded, not necessarily self-adjoint and is assumed to generate an analytic semigroup $S(t):=e^{tA}$. Note that $\partial^{\alpha}_t$ denotes the Caputo fractional derivative with $\alpha\in(\frac 34,1)$,  $W(t)=W(x,t)$ is a $H$-valued $Q$-Wiener process defined in a filtered probability space $(\Omega, \mathcal{F}, \mathbb{P},\{\mathcal{F}_t\}_{t\geq 0})$, where the covariance operator $Q: H\rightarrow H$ is a positive and linear self-adjoint operator. The filtration is assumed to fulfil the usual assumptions (see \cite[Def 2.1.11]{Prevot}). The $Q$-Wiener process $W(t)$ can be represented as follows \cite{Prevot}
	\begin{equation}
		\label{def_Brow}
		W(x,t)=\sum_{i=0}^{\infty}\beta_i(t)Q^{\frac 12}e_i(x)=\sum_{i=0}^{\infty}\sqrt{q_i}\beta_i(t)e_i(x).
	\end{equation}
	where $q_i,\,e_i,\,i\in\N$ are respectively the eigenvalues and eigenfunctions of the covariance operator $Q$, and $\beta_i$ are mutually independent and identically distributed standard normal distributions. The mark set $\mathcal{X}$ is defined by $\mathcal{X}:=H-\{0\}$. For a given set $\Gamma$, we denote by $\mathcal{B}(\Gamma)$ the smallest $\sigma$-algebra containing all open sets of $\Gamma$. Let $\left(\mathcal{X},\mathcal{B}(\mathcal{X}),\upsilon\right)$ be a $\sigma$-finite measurable space and  $\upsilon$ (with $\upsilon\neq 0$) a Levy measurable on $\mathcal{B}(\mathcal{X})$ such that
	\begin{equation}
		\upsilon(\{0\})=0,\,\,\,\,\,\,\,\,\,\,\int_{\mathcal{X}}\min(\left\|z\right\|^2,1)\upsilon(dz)<\infty.
	\end{equation}
	Let $N(dz,dt)$ be the $H$-valued Poisson distributed $\sigma$-finite measure on the product $\sigma$-algebra $\mathcal{B}(\mathcal{X})$ and $\mathcal{B}(\R_+)$ with intensity $\upsilon(dz)dt$, where $dt$ is the Lebesque measure on $\mathcal{B}(\R_+)$. In our model problem \eqref{model_timfrac}, $\widetilde{N}(dz,dt)$ stands for the compensated Poisson random measure defined by
	\begin{equation}
		\widetilde{N}(dz,dt):=N(dz,dt)-\upsilon(dz)dt.
	\end{equation}
	Note that $\widetilde{N}(dz,dt)$ is a noncontinuous martingale with mean 0 (see e.g \cite{Man}). The Wiener process $W$ and the compensated Poisson measure $\widetilde{N}$ are supposed to be independent. Precise assumptions on the nonlinear functions $F$, $B$ and $G$ to ensure the existence of the mild solution of \eqref{model_timfrac} will be given in the following section.
	
	In the last few decades, fractional calculus has become of increasing interest to researchers in various fields of science and technology. The theory of fractional partial differential equations has gained considerable interest over time, and since most of these equations have no analytical solutions, numerical schemes are the only  tools  to provide good approximations.
	For deterministic  equation ($G=B=0$) and  self adjoint operator $A$, Lin and Xu \cite{Lin}  have considered  the numerical approximation of time-fractional diffusion equation and proposed an algorithm  based on the  finite difference scheme in time and Legendre spectral method in space. In the same context,  high order finite element method and mixed finite element scheme have been studied in \cite{Jia,Liu}.  The numerical methods to solve the fractional heat equation with Dirichlet condition, involving a Riemann Liouville fractional derivative in time has presented in \cite{For,Pri}.  Gao et al. \cite{Gao} presented a  novel fractional numerical method (called $L_{1-2}$ formula) to approximate the Caputo fractional derivative order $\alpha$ ($0<\alpha<1$) with a modification of the classical $L_1$ formula and prove that the computational efficiency and  the numerical accuracy of the new formula are superior than the standard  $L_1$ formula. The Galerkin finite element approximation for time-fractional Navier-Stokes and of  the semilinear time-fractional subdiffusion problem is studied in \cite{Lia,Alm}. In \cite{Osm}, the authors developed an alternative numerical method based upon the Keller Box method for the subdiffusion equation and in \cite{Elz} Elzaki et al. used the decomposition method coupled with Elzaki transform to construct appropriate solutions to multi-dimensional wave, Burger and Klein-Gordon equations of fractional order. The authors in \cite{Yan} presented a new type of discrete fractional Gronwall inequality and they used it to analyse the stability and  the convergence of  the Galerkin spectral method for a linear time-fractional subdiffusion equation. 
	Note that the time stepping methods used in all the works mentioned until now  are based on finite difference methods. However theses schemes are explicit, but unstable, unless the time stepsize is very small.  To solve that drawback, numerical method based on exponential integrators of Adams type have  been  proposed in \cite{ExpF}. The price to pay is the computation of Mittag-Leffler (ML) matrix functions. As ML matrix function is the generalized form of the exponential of matrix function, works in \cite{ML1,ML2,ML3} have extended  some exponential  computational techniques to ML. Note that up to  now all the numerical algorithms presented are for  time fractional deterministic PDEs with self adjoint  linear operators.
	
	However in  order to represent real-world physical phenomena more accurately, it is necessary to take into account stochastic disturbances from uncertain input data.  The uncertain  is usually  modelled by  including  the  standard Brownian motion (Gaussian noise)  and the corresponding model equation is given by \eqref{model_timfrac} with $G=0$.  Few works have been done for numerical methods for Gaussian noise and time fractional stochastic partial differential equation \eqref{model_timfrac} with $G=0$, even when  the linear operator $A$ is self adjoint. To the best of our knowledge, \cite{Zoua} is the first of basic theory and numerical method for a class of these fractional SPDEs. 
	In \cite{Zoub}, the authors developed the fully discrete Galerkin finite element method for solving the time-fractional stochastic diffusion \footnote{So the corresponding linear operator is self adjoint.} equations  based on the approximations of the Mittag-Leffler function. Indeed the temporal integration is similar to the  deterministic exponential scheme  in \cite{ExpF}.
	In \cite{Gun}, the authors  provided  rigorous convergence of numerical method for solving the stochastic time-fractional partial differential equation where  the 
	temporal discretization is done by the backward-Euler convolution quadrature.  Note that all the above works have been done for self adjoint linear operator $A$, so numerical  study  for \eqref{model_timfrac} with $G=0$ and  non  self adjoint operator $A$
	is still an open problem in the field, to the best of our knowledge.

	Furthermore  in finance for example,  the unpredictable nature of many events such as market crashes, announcements made by the central banks, changing credit risk, insurance in a changing risk, changing face of operational risk \cite{Platen1,Tankov} might have sudden and significant impacts on the stock price. In such situation, the more  realistic model  is built by  incorporating a non-Gaussian noise such as L\'evy process or Poisson random measure to model such events. The corresponding equation is our model equation given in \eqref{model_timfrac}.  As we have mentioned,  numerical schemes for such SPDE of type \eqref{model_timfrac} driven by Gaussian and non Gaussian noises have been lacked in the scientific literature,  our goal will be to fill that gap by extending the exponential scheme \cite{Lor,Nou} to time-fractional  SPDE of type \eqref{model_timfrac}. The extension is extremely complicated since the ML function is more  challenging than the exponential function. Using novel technical results that we have developed here, we have provided the strong convergence of our full discrete scheme for \eqref{model_timfrac}.
	Our strong convergence results examine how the convergence orders depend on the regularity of the initial data and the power of fractional derivative.  

	The rest of the paper is structured as follows. In Section \ref{mathsetting}, Mathematical settings for cylindrical Brownian motion, random Poisson measure,  Caputo-type fractional derivative, Laplace transform, Mainardi's Wright-type function are presented, along with the well posedness and regularities results of the mild solution of SPDE \eqref{model_timfrac}. In Section \ref{reg}, numerical schemes based on stochastic exponential  integrator scheme for SPDE \eqref{model_timfrac} are presented. We give some regularity estimates of the semi-discrete problem and analyse the spatial error in Section \ref{spa_conv}. We end the paper in Section \ref{schemes}, by presenting  the strong convergence proof  of the full scheme for  \eqref{model_timfrac} based on finite element for spatial discretization and exponential integrator for temporal discretization.

	\section{Mathematical setting, main assumptions and well posedness problem}
	\label{mathsetting}
	In this section, some notations and preliminary results needed throughout this  work are provided. Let $\left(K,\langle .,.\rangle_K,\|.\|\right)$ be a separable Hilbert space. For $p\geq 2$ and for a Banach space U, we denote by $L^p(\Omega,U)$ the Banach space of $p$-integrable $U$-valued random variables. 
	We denote by $L(U,K)$ the space of bounded linear mapping from $U$ to $K$ endowed with the usual operator norm $\|.\|_{L(U,K)}$ and $\mathcal{L}_2(U,K)=HS(U,K)$ the space of Hilbert-Schmidt operators from $U$ to $K$ equipped  with the following norm
	\begin{eqnarray}
		\label{normL2UK}
		\left\|l\right\|_{\mathcal{L}_2(U,K)}:=\left(\sum_{i\in\N^d}\|l\psi_i\|^2\right)^{\frac 12},\,\,\,\,\,l\in \mathcal{L}_2(U,K),
	\end{eqnarray}
	where $(\psi_i)_{i\in\N^d}$ is an orthonormal basis on $U$. The sum in $\eqref{normL2UK}$ is independent of the choice of the orthonormal basis of $U$. We use the notation $L(U,U)=:L(U)$ and $\mathcal{L}_2(U,U)=:\mathcal{L}_2(U)$. It is well known that for all $l\in L(U,K)$ and $l_1\in \mathcal{L}_2(U)$, $ll_1\in \mathcal{L}_2(U,K)$ and 
	\begin{eqnarray*}
		\|ll_1\|_{\mathcal{L}_2(U,K)}\leq \left\|l\right\|_{L(U,K)}\|l_1\|_{\mathcal{L}_2(U)}.
	\end{eqnarray*} 
	We denote by $L^0_2:=HS(Q^{\frac 12}(H),H)$ the space of Hilbert-Schmidt operators from $Q^{\frac 12}(H)$ to $H$ with corresponding norm $\|.\|_{L^0_2}$ defined by 
	\begin{eqnarray}
		\label{normL02}
		\|l\|_{L_2^0}:=\left\|lQ^{\frac 12}\right\|_{HS}=\left(\sum_{i\in\N^d}\|lQ^{\frac 12}e_i\|^2\right)^{\frac 12},\,\,\,\,\,l\in L_2^0,
	\end{eqnarray}
	where $(e_i)_{i\in\N^d}$ is an orthonormal basis of $H$. The sum in \eqref{normL02} is also independent of the choice of the orthonormal basis of $H$. Let $L^2_{\nu}(\chi\times[0,T];H)$ be the space of all mappings $\theta: \chi\times[0,T]\times \Omega\rightarrow H$ such that $\theta$ is jointly measurable and $\mathcal{F}_t$-adapted for all $z\in\chi$, $0\leq s\leq T$ satisfying
	\begin{eqnarray*}
		\int_0^T\int_{\chi}\left\|\theta(z,s)\right\|^2\nu(dz)ds<\infty.
	\end{eqnarray*}
	The following lemma is a result that will be used throughout this paper.
	\begin{lem}(It\^{o} Isometry: \cite[(4.30)]{Pra}, \cite[(3.56)]{Man})
		\begin{enumerate}
			\item [(i)] Let $\phi\in L^2([0,T];L^0_2)$, then the following holds
			\begin{eqnarray}
				\label{browint}
				\mathbb{E}\left[\left\|\int_0^T\phi(s)dW(s)\right\|^2\right]=\mathbb{E}\left[\int_0^T\left\|\phi(s)\right\|^2_{L^0_2}ds\right].
			\end{eqnarray}
			\item [(ii)] Let $\theta\in L^2_{\nu}(\chi\times[0,T];\mathcal{H})$, then the following holds
			\begin{eqnarray}
				\label{jumpint}
				\mathbb{E}\left[\left\|\int_0^T\int_{\chi}\theta(z,s)\tilde{N}(dz,ds)\right\|^2\right]=\mathbb{E}\left[\int_0^T\int_{\chi}\left\|\theta(z,s)\right\|^2\nu(dz)ds\right].
			\end{eqnarray}
		\end{enumerate}
	\end{lem}
	\begin{dfn}
		\label{cap}
		The Caputo-type derivative of order $\alpha$ with respect to $t$ is defined by
		\begin{eqnarray}
			\label{cap_der}
			\partial^{\alpha}_tX(t)=	
			\left\{
			\begin{array}{ll}
				\frac{1}{\Gamma(1-\alpha)}\int_{0}^t\frac{\partial X(s)}{\partial s}\frac{ds}{(t-s)^{\alpha}},\,\,\,\,0<\alpha<1 \\
				\frac{\partial X}{\partial t},\,\,\,\,\alpha=1,
			\end{array}
			\right.
		\end{eqnarray}
		where $\Gamma(\cdot)$ is the gamma function.
	\end{dfn}
	Let introduce the generalized Mittag-Leffler function $E_{\alpha,\beta}(t)$ defined as follows:
	\begin{eqnarray}
		\label{mit}
		E_{\alpha,\beta}(t)=\sum_{k=0}^{\infty}\frac{t^k}{\Gamma(\alpha k+\beta)},
	\end{eqnarray}
	and his Laplace transform given by (see \cite{Hau})
	\begin{eqnarray}
		\label{lap_mit}
		\mathcal{L}(t^{\beta-1}E_{\alpha,\beta}(\lambda t^{\alpha}))=\int_{0}^{\infty}\exp^{-\varsigma t}t^{\beta-1}E_{\alpha,\beta}(\lambda t^{\alpha})dt=\frac{\varsigma^{\alpha-\beta}}{\varsigma^{\alpha}-\lambda}.
	\end{eqnarray}
	Remember that  the Laplace transform  is defined by
	\begin{eqnarray}
		\label{lap}
		\tilde{f}(\varsigma)=\mathcal{L}(f(t))=\int_{0}^{\infty}e^{-\varsigma t}f(t)dt.
	\end{eqnarray} 
	Now we give the defintion of mild solution to \eqref{model_timfrac}.
	\begin{dfn}( \cite[Definition 2.2]{Zoua})
		\label{mildsol}
		For any $0<\alpha<1$, a stochastic process $\{X(t), t\in[0,T]\}$ is called mild solution of \eqref{model_timfrac} if 
		\begin{enumerate}
			\item [1.] $X(t)$ is $\mathcal{F}_t$-adapted on the filtration $(\Omega, \mathcal{F}, \mathbb{P},\{\mathcal{F}_t\}_{t\geq 0})$,
			\item [2.] $\{X(t), t\in[0,T]\}$ is measurable and $\mathbb{E}\left[\int_{0}^T\left\|X(t)\right\|^2dt\right]<\infty$,
			\item [3.] For all $t\in[0,T]$, 
			\begin{eqnarray}
				\label{sol}
				X(t)&=&S_1(t)X_0+\int_{0}^{t}(t-s)^{\alpha-1}S_2(t-s)F(X(s))ds\nonumber\\
				&+&\int_{0}^{t}(t-s)^{\alpha-1}S_2(t-s)B(X(s))dW(s)\nonumber\\
				&+&\int_{0}^{t}\int_{\mathcal{X}}(t-s)^{\alpha-1}S_2(t-s)G(z,X(s))\widetilde{N}(dz,ds). 
			\end{eqnarray}
			hold a.s. Where $S_1(t)=E_{\alpha,1}(At^{\alpha})$ and $S_2(t)=E_{\alpha,\alpha}(At^{\alpha})$.
		\end{enumerate}
	\end{dfn}
	\begin{rem}
		Considering the Mainardi's Wright-type function (see \cite{Mai})
		\begin{eqnarray*}
			M_{\alpha}(\theta)=\sum_{n=0}^{\infty}\frac{(-1)^n\theta^n}{n!\Gamma(1-\alpha(1+\theta))},\quad 0<\alpha<1,\quad \theta>0,
		\end{eqnarray*}
		then the following results  holds 
		\begin{eqnarray}
			\label{Mai1}
			M_{\alpha}(\theta)\geq 0,\quad \int_{0}^{\infty}\theta^{\mu}M_{\alpha}(\theta)d\theta=\frac{\Gamma(1+\mu)}{\Gamma(1+\alpha\mu)},\quad-1<\mu<\infty,\quad\theta>0,
		\end{eqnarray}
		and
		\begin{eqnarray}
			\label{Mai2}
			E_{\alpha,1}(t)=\int_0^{\infty}M_{\alpha}(\theta)e^{t\theta}d\theta,\,\,\,\,\,E_{\alpha,\alpha}(t)=\int_0^{\infty}\alpha\theta M_{\alpha}(\theta)e^{t\theta}d\theta.
		\end{eqnarray}
	\end{rem}
	Using \eqref{Mai2}, we rewrite operator $S_1(t)$ and $S_2(t)$ as following
	\begin{eqnarray}
		\label{S1}
		S_1(t)=E_{\alpha,1}(At^{\alpha})=\int_0^{\infty}M_{\alpha}(\theta)e^{A\theta t^{\alpha}}d\theta=\int_0^{\infty}M_{\alpha}(\theta)S(\theta t^{\alpha})d\theta,
	\end{eqnarray}
	and 
	\begin{eqnarray}
		\label{S2}
		S_2(t)=E_{\alpha,\alpha}(At^{\alpha})=\int_0^{\infty}\alpha\theta M_{\alpha}(\theta)e^{A\theta t^{\alpha}}d\theta=\int_0^{\infty}\alpha\theta M_{\alpha}(\theta)S(\theta t^{\alpha})d\theta.
	\end{eqnarray}
	Combining \eqref{sol}, \eqref{Mai1} and \eqref{Mai2}, we obtain the result presented in \cite[Definition 2.5]{Zha} and we have the following lemma.
	\begin{lem}(\cite[Lemma 2.8]{Zha})
		\label{bound}
		The operators $\{S_1(t)\}_{t\geq 0}$ and $\{S_2(t)\}_{t\geq 0}$ depending of $\alpha\in(\frac 34,1)$ are bounded linear operator  and the following estimates hold
		\begin{eqnarray}
			\label{prop_semi}
			\left\|S_1(t)v\right\|\leq C_1 e^{-\gamma t}\left\|v\right\|,\quad \left\|S_2(t)v\right\|\leq \frac{C_1\alpha}{\Gamma(1+\alpha)}e^{-\gamma t}\left\|v\right\|,\quad v \in H,
		\end{eqnarray}
		and for some constants $C_1, \gamma>0$.
	\end{lem}
	In order to ensure the existence and the uniqueness of mild solution for SPDE \eqref{model_timfrac} and for the purpose of convergence analysis we make the following assumptions.
	\begin{ass}[Initial Value]
		\label{init}
		We assume that the initial data $X_0:\Omega\rightarrow H$ to be $\mathcal{F}_0$-measurable mapping and $X_0\in L^2(\Omega, D((-A)^{\frac{\beta}{2}}))$ with $0\leq\beta<2$.
	\end{ass}
	\begin{ass}[Non linearity term F]
		\label{nonlin}
		We assume the nonlinear mapping $F:H\rightarrow H$, to be linear growth and Lipschitz continuous ie there exist constant $L>0$ such that 
		\begin{eqnarray}
			\label{lip}
			\|F(u)-F(v)\|^2\leq L\|u-v\|^2,\quad \|F(v)\|^2\leq L(1+\|v\|^2),\quad u,v\in H.
		\end{eqnarray}
	\end{ass}
	\begin{ass}[Lipschitz condition]
		\label{lips}
		We assume that the diffusion and jump coefficients $B:H\rightarrow L^0_2$ and $G:\mathcal{X}\times H\rightarrow H$ satisfy the global Lipschitz condition ie, there exists a positive constant $L>0$ such that: 
		\begin{eqnarray}
			&&\|B(u)-B(v)\|_{L^0_2}^2\leq L\|u-v\|^2,\nonumber\\ &&\int_{\mathcal{X}}\|G(z,u)-G(z,v)\|^2\upsilon(dz)\leq L\|u-v\|^2,\quad u,v\in H.
		\end{eqnarray}
	\end{ass} 
	\begin{ass}[Linear growth]
		\label{growth}
		For $\tau\in[0,1)$ and some constant $L>0$ the following bound holds
		\begin{eqnarray}
			&&\|(-A)^{\tau}B(u)\|_{L^0_2}^2\leq L(1+\|(-A)^{\tau}u\|^2),\nonumber\\
			&&\int_{\mathcal{X}}\|(-A)^{\tau}G(z,u)\|^2\upsilon(dz)\leq L(1+\|(-A)^{\tau}u\|^2),\hspace{0.25cm}u\in H.
		\end{eqnarray}
	\end{ass} 
	\begin{thm}(\cite[Corollary 3.2]{Zha})
		\label{well_poss}
		Under the Assumptions \ref{init}-\ref{growth}, if \\  $\frac{9}{2}\left(\frac{C_1\alpha}{\Gamma(1+\alpha)}\right)^2\frac L{\gamma}\left(\frac 1{2\alpha-1}\right)^{\frac 12}<1$ the SPDEs \eqref{model_timfrac} admits a unique mild solution $X(t)\in(D[0,T],H)$ asymptotic stable in mean square, that is
		\begin{eqnarray}
			\label{boun}
			\mathbb{E}\left[\sup_{0\leq t\leq T}\left\|X(t)\right\|^2\right]<\infty,
		\end{eqnarray}
		where by $(D[0,T],H)$ we denote the space of all adapted c\`adl\`ag processes defined on $[0,T]$ with values in $H$.  Note that $C_1$ and $\gamma$ are given in  \eqref{prop_semi}, and  $L$ is the  Lipschitz condition from  Assumptions \ref{nonlin}-\ref{growth},
	\end{thm}
	In all that follows, $C$ denotes a positive constant that may change from line to line. In the Banach space $D((-A)^{\frac{\alpha}{2}})$, $\alpha\in\R$, we use  the notation $\|(-A)^{\frac{\alpha}{2}}\cdot\|=\|\cdot\|_{\alpha}$ and we now present the following regularity results.

	\section{Regularity of the mild solution}
	\label{reg}
	We discuss the space and regularity of the mild solution $X(t)$ of \eqref{model_timfrac} given by \eqref{sol} in this section. In the rest of this paper to simplify the presentation, we assume the SPDE \eqref{model_timfrac} to be second order of the following type.
	\begin{eqnarray*}
		\label{model}
		\partial_t^{\alpha}X(t)&=&\left[\nabla\cdot(D\nabla X(t,x))-q\cdot \nabla X(t,x)+f(x,X(t,x))\right]\nonumber\\
		&+&b(x,X(t,x))\frac{dW(t,x)}{dt}+\frac{\int_{\mathcal{X}}g(z,x,X(t,x))\widetilde{N}(dz,dt)}{dt},
	\end{eqnarray*}
	where $f:\Lambda\times\R\rightarrow \R$ is globally continuous,  $b:\Lambda\times\R\rightarrow \R$ is continuously differentiable with globally bounded derivatives and  $g:\mathcal{X}\times\Lambda\times\R\rightarrow \R$ is globally Lipschitz continuous. In the abstract framework \eqref{model_timfrac}, the linear operator A  is the $L^2(\Lambda)$ realization (see \cite[p. 812]{Fuj}) of  the following  differential operator 
	\begin{eqnarray*}
		&& \mathcal{A}u=\sum_{i,j=1}^d\frac{\partial}{\partial x_i}\left(D_{i,j}(x)\frac{\partial u}{\partial x_j}\right)-\sum_{i=1}^d q_i(x)\frac{\partial u}{\partial x_i},\\
		&&D=\left(D_{i,j}\right)_{1\leq i,j\leq d},\hspace{1cm}q=\left(q_i\right)_{1\leq i\leq d},
	\end{eqnarray*}
	where $D_{i,j}\in L^{\infty}(\Lambda)$, $q_i\in L^{\infty}(\Lambda)$. We assume that there exists a positive constant $c_1>0$ such that
	\begin{eqnarray*}
		\sum_{i,j=1}^d D_{i,j}(x)\xi_i\xi_j\geq c_1|\xi|^2,\hspace{1cm}\xi\in\R^d,\hspace{0.5cm}x\in\bar{\Lambda}.
	\end{eqnarray*}
	The functions $F:H\rightarrow H$, $B:H\rightarrow L^0_2$ and $G:\mathcal{X}\times H\rightarrow H$ are defined by 
	\begin{eqnarray*}
		&&(F(v))(x)=f(x,v(x)),\quad (B(v)u)(x)=b(x,v(x))\cdot u(x),\\
		&&G(z,v)(x)=g(z,x,v(x)).
	\end{eqnarray*}
	for all $x\in\Lambda$, $v\in H$, $u\in Q^{1/2}(H)$ and $z\in \mathcal{X}$. As in \cite{Fuj,Lor}, we introduce two spaces $\mathbb{H}$, and $V$ such that $\mathbb{H}\subset V$; the two spaces depend on the boundary conditions of $\Lambda$ and the domain of the operator A. For Dirichlet (or first-type) boundary conditions, we take 
	\begin{equation*}
		V=\mathbb{H}=H^1_0(\Lambda)=\{v\in H^1(\Lambda):v=0\,\,\text{on}\,\,\partial \Lambda\}.
	\end{equation*}
	For Robin (third-type) boundary condition and Neumann (second-type) boundary condition, which is a special case of Robin boundary condition, we take $V=H^1(\Omega)$ 
	\begin{equation*}
		\mathbb{H}=\{v\in H^2(\Lambda): \partial v/\partial v_{\mathcal{A}}+\alpha_0 v=0,\hspace{0.5cm}\text{on} \hspace{0.5cm}\partial\Lambda\},\hspace{1cm}\alpha_0\in\R.
	\end{equation*}
	Where $\partial v/\partial v_{\mathcal{A}}$ is the normal derivative of $v$ and $v_{\mathcal{A}}$ is the exterior pointing normal $n=(n_i)$ to the boundary of $\mathcal{A}$ given by
	\begin{equation*}
		\partial v/\partial v_{\mathcal{A}}=\sum_{i,j=1}^d n_i(x)D_{i,j}(x)\frac{\partial v}{\partial x_j},\,\,\,\,\,x\in\partial\Lambda.
	\end{equation*}
	Using  G\aa rding's inequality (see e.g. \cite{ATthesis}), it holds that there exist two constants $c_0$ and $\lambda_0>0$ such that   the bilinear form $a(.,.)$ associated to $-A$ satisfies
	\begin{eqnarray}
		\label{ellip1}
		a(v,v)\geq \lambda_0\Vert v \Vert^2_{H^1(\Lambda)}-c_0\Vert v\Vert^2, \quad v\in V.
	\end{eqnarray}
	By adding and substracting $c_0Xdt$ in both sides of \eqref{model}, we have a new linear operator
	still denoted by $A$, and the corresponding  bilinear form is also still denoted by $a$. Therefore, the following coercivity property holds
	\begin{eqnarray}
		\label{ellip2}
		a(v,v)\geq \lambda_0\Vert v\Vert^2_1,\quad v\in V.
	\end{eqnarray}
	Note that the expression of the nonlinear term $F$ has changed as we included the term $c_0X$ in the new nonlinear term that we still denote by  $F$. 
	The coercivity property \eqref{ellip2} implies that  $A$ is the infinitesimal generator of a contraction semigroup $S(t)=e^{t A}$  on $L^{2}(\Lambda)$. 
	Note also  that the coercivity  property \eqref{ellip2} also implies that   the real part of the eigenvalues of   $-A$  are positive,   therefore  its fractional powers are well defined  for any $\alpha>0,$ by
	\begin{equation}
		\label{fractional}
		\left\{\begin{array}{rcl}
			( -A)^{-\alpha} & =& \frac{1}{\Gamma(\alpha)}\displaystyle\int_0^\infty  t^{\alpha-1}{\rm e}^{tA}dt,\\
			(-A)^{\alpha} & = & (-A^{-\alpha})^{-1},
		\end{array}\right.
	\end{equation}
	where $\Gamma(\alpha)$ is the Gamma function.  As  the real part of the eigenvalues of  $A$  is negative, the generalized Mittag-Leffler operator  $E_{\alpha,\beta}(tA)$ is therefore well defined  by
	\begin{eqnarray*}
		\Vert (\lambda I +A )^{-1} \Vert_{L(L^{2}(\Lambda))} \leq \dfrac{C_{1}}{\vert \lambda \vert },\;\quad \quad
		\lambda \in S_{\theta},
	\end{eqnarray*}
	where $S_{\theta}:=\left\lbrace  \lambda \in \mathbb{C} :  \lambda=\rho e^{i \phi},\; \rho>0,\;0\leq \vert \phi\vert \leq \theta \right\rbrace $ (see e.g. \cite{Henry}).
	
	We recall the following properties of the semigroup $S(t)$ generated by $-A$, that will be useful throughout this paper.
	\begin{prp}[Smoothing properties of the semigroup]\cite{Pazy}
		\label{semigroup}
		Let $\alpha>0$, $\delta\geq 0$ and $0\leq \gamma\leq 1$, then there exists a constant $C>0$ such that 
		\begin{eqnarray}
			\label{semigroup_prp1}
			&&\|(-A)^{\delta}S(t)\|_{L(H)}\leq C t^{-\delta},\quad\|(-A)^{-\gamma}(I-S(t))\|_{L(H)}\leq C t^{\gamma},\quad t>0,\\
			\label{semigroup_prp2}
			&&(-A)^{\delta}S(t)=S(t)(-A)^{\delta}\, \text{on}\, D(A^{\delta}),\,\text{If} \,\delta>\gamma\,\text{then}\,D((-A)^{\delta})\supset D((-A)^{\gamma}).
		\end{eqnarray}
	\end{prp}
	Firstly, we have the following useful lemma
	\begin{lem}
		\label{app}
		Let $0\leq t_1<t_2\leq T$ and $0< a< 1$, we have
		\begin{eqnarray}
			t_2^a-t_1^a\leq (t_2-t_1)^a.
		\end{eqnarray}
		\textit{Proof:} Using the integral form and the variable change $t=t_2 u+t_1(1-u)$ yields
		\begin{eqnarray*}
			t_2^a-t_1^a&=&\int_{t_1}^{t_2}a\,t^{a-1}dt\\
			&=& a(t_2-t_1)\int_{0}^{1}\left[t_1+(t_2-t_1)u\right]^{a-1}du\\
			&\leq& a(t_2-t_1)\int_{0}^{1}(t_2-t_1)^{a-1}u^{a-1}du\\
			&\leq& (t_2-t_1)^a\int_{0}^{1}a\,u^{a-1}du\\
			&\leq& (t_2-t_1)^a.
		\end{eqnarray*}
	\end{lem}
	The following lemma is an extension of the smoothing properties of the semigroup $S(t)$, \propref{semigroup} to the operators $\{S_1(t)\}_{t\geq 0}$ and $\{S_2(t)\}_{t\geq 0}$.
	\begin{lem}
		\label{semigroup12}
		Let $t\in(0,T)$, $0<t_1<t_2\leq T$, $T<\infty$,  $\frac 12<\alpha<1$, $\rho\geq 0$, $0\leq \eta<1$ and  $\delta\geq 0$, there exists a constant $C>0$ such that for all $i=1,2$
		\begin{eqnarray}
			\label{semigroup_prp}
			\|(-A)^{\rho}S_i(t)\|_{L(H)}\leq C t^{-\alpha\rho},\hspace{0.5cm}\|(-A)^{-\eta}\left(S_1(t_2)-S_1(t_1)\right)\|_{L(H)}\leq C t^{\alpha\eta},
		\end{eqnarray}
		\begin{eqnarray}
			\label{semigroup_prp1*}
			\left\|t_1^{\alpha-1}S_2(t_1^{\alpha})-t_2^{\alpha-1}S_2(t_2^{\alpha})\right\|_{L(H)}\leq C (t_2-t_1)^{1-\alpha}t_1^{\alpha-1}t_2^{\alpha-1},
		\end{eqnarray}
		and
		\begin{eqnarray}
			\label{semigroup_prp*}
			(-A)^{\delta}S_i(t)=S_i(t)(-A)^{\delta}\hspace{0.25cm}\text{on}\hspace{0.25cm}D((-A)^{\delta}).
		\end{eqnarray}
	\end{lem}
	\begin{proof}
		See \cite[Lemma 3.3]{Zoua} for the proof of the first result of \eqref{semigroup_prp} and \eqref{semigroup_prp*} is just a consequence of \eqref{semigroup_prp2} using \eqref{Mai1} and \eqref{Mai2}. Concerning the second, using triangle inequality, Proposition \ref{semigroup}, \eqref{Mai1} and Lemma \ref{app}, we have
		\begin{eqnarray*}
			\left\|(-A)^{-\eta}\left(S_1(t_2)-S_1(t_1)\right)\right\|_{L(H)}&=&\left\|\int_{0}^{\infty}(-A)^{-\eta}M_{\alpha}(\theta)\left(S(\theta t_2^{\alpha})-S(\theta t_1^{\alpha})\right)d\theta\right\|_{L(H)}\\
			&\leq&\int_{0}^{\infty}M_{\alpha}(\theta)\left\|S(\theta t_1^{\alpha})\right\|_{L(H)}\left\|(-A)^{-\eta}\left(e^{A\theta (t_2^{\alpha}-t_1^{\alpha})}-I\right)\right\|_{L(H)}d\theta\\
			&\leq&C\int_{0}^{\infty}\left[\theta\,(t_2^{\alpha}-t_1^{\alpha})\right]^{\eta} M_{\alpha}(\theta)d\theta\\
			&\leq& C(t_2^{\alpha}-t_1^{\alpha})^{\eta}\int_{0}^{\infty}\theta^{\eta} M_{\alpha}(\theta)d\theta\\
			&\leq& C\frac{\Gamma(1+\eta)}{\Gamma(1+\alpha\eta)}(t_2-t_1)^{\alpha\eta}\leq  C\,(t_2-t_1)t^{\alpha\eta}.
		\end{eqnarray*}
		For the proof of \eqref{semigroup_prp}, we use triangle inequality, Proposition \ref{semigroup} and \eqref{Mai1} to obtain
		\begin{eqnarray*}
			&&\left\|t_2^{\alpha-1}S_2(t_2)-t_1^{\alpha-1}S_2(t_1)\right\|_{L(H)}\\
			&=&\left\|\int_0^{\infty}\alpha\theta M_{\alpha}(\theta)\left[ t_2^{\alpha-1}e^{A\theta t_2^{\alpha}}-t_1^{\alpha-1}e^{A\theta t_1^{\alpha}}\right] d\theta\right\|_{L(H)}\\
			&=&\left\|\int_0^{\infty}\alpha\theta M_{\alpha}(\theta)\left(\int_{t_1}^{t_2} \frac{d(t^{\alpha-1}e^{A\theta t^{\alpha}})}{dt}dt\right) d\theta\right\|_{L(H)}\\
			&=&\left\|\int_0^{\infty}\alpha\theta M_{\alpha}(\theta)\left(\int_{t_1}^{t_2} (\alpha-1)t^{\alpha-2}e^{A\theta t^{\alpha}}+\alpha\theta At^{2\alpha-2}e^{A\theta t^{\alpha}} dt\right)d\theta\right\|_{L(H)}\\
			&\leq&\int_0^{\infty}\alpha\theta M_{\alpha}(\theta)\left(\int_{t_1}^{t_2} (1-\alpha)t^{\alpha-2}\left\|e^{A\theta t^{\alpha}}\right\|_{L(H)}+\alpha\theta t^{2\alpha-2}\left\|(-A)^{1}e^{A\theta t^{\alpha}}\right\|_{L(H)} dt\right) d\theta\\
			&\leq& C\int_0^{\infty}\alpha\theta M_{\alpha}(\theta)\left(\int_{t_1}^{t_2} (1-\alpha)t^{\alpha-2}+\alpha\theta t^{2\alpha-2}(\theta t^{\alpha})^{-1} dt\right) d\theta\\
			&\leq& C\int_0^{\infty}\alpha\theta M_{\alpha}(\theta)\left(\int_{t_1}^{t_2} t^{\alpha-2} dt\right) d\theta\\
			&\leq& C\frac{\alpha}{1-\alpha}\int_0^{\infty}\theta M_{\alpha}(\theta)\left( t_1^{\alpha-1}-t_2^{\alpha-1}\right) d\theta\\
			&\leq& C\frac{\alpha\Gamma(2)}{(1-\alpha)\Gamma(1+\alpha)}\frac{t_2^{1-\alpha}-t_1^{1-\alpha}}{t_1^{1-\alpha}t_2^{1-\alpha}},
		\end{eqnarray*}
		applying \lemref{app} with $a=1-\alpha$ yields
		\begin{eqnarray*}
			\left\|t_2^{\alpha-1}S_2(t_2)-t_1^{\alpha-1}S_2(t_1)\right\|_{L(H)}&\leq& C\frac{\alpha\Gamma(2)}{(1-\alpha)\Gamma(1+\alpha)}(t_2-t_1)^{1-\alpha}t_1^{\alpha-1}t_2^{\alpha-1}\\
			&\leq& C(t_2-t_1)^{1-\alpha}t_1^{\alpha-1}t_2^{\alpha-1} .
		\end{eqnarray*}
	\end{proof}
	Moreover \eqref{semigroup_prp}-\eqref{semigroup_prp*} hold if $A$, $S_1$ and $S_2$ are replaced by their discrete versions $A_h$, $S_{1h}$ and $S_{2h}$ respectively defined in Section \ref{spa_conv}. 
	
	Now, we give a spatial regularity result for the solution $X(t)$ in the following lemma.
	\begin{lem}
		\label{spa_reg}
		Let Assumptions \ref{init}- \ref{growth} be fulfilled. Then the following space regularity holds
		\begin{eqnarray}
			\label{spa_reg1}
			\|(-A)^{\beta/2}X(t)\|_{L^2(\Omega,H)}\leq C(1+\|(-A)^{\beta/2}X_0\|_{L^2(\Omega,H)}),\quad 0\leq t\leq T,
		\end{eqnarray}
	\end{lem}
	Moreover, \eqref{spa_reg1} holds if $A$ and $X$ are replaced by their discrete versions $A_h$ and $X^h$ defined in Section \ref{spa_conv}.\\
	\begin{proof}
		By the definition of mild solution \eqref{sol},
		\begin{eqnarray*}
			X(t)&=&S_1(t)X_0+\int_{0}^{t}(t-s)^{\alpha-1}S_2(t-s)F(X(s))ds\\
			&+&\int_{0}^{t}(t-s)^{\alpha-1}S_2(t-s)B(X(s))dW(s)\nonumber\\
			&+&\int_{0}^{t}\int_{\mathcal{X}}(t-s)^{\alpha-1}S_2(t-s)G(z,X(s))\widetilde{N}(dz,ds). 
		\end{eqnarray*}
		Then taking the $L^2$-norm, using triangle inequality, the classical estimate $\left(\sum_{i=1}^na_i\right)^2\leq n\sum_{i=1}^na_i^2$, the It\^o isometry \eqref{browint} and \eqref{jumpint} to the last two terms yield
		\begin{eqnarray}
			\label{spareg}
			&&\left\|(-A)^{\beta/2}X(t)\right\|^2_{L^2(\Omega,H)}\nonumber\\ &\leq&4\left\|(-A)^{\beta/2}S_1(t)X_0\right\|^2_{L^2(\Omega,H)}+4\left(\int_{0}^{t}(t-s)^{\alpha-1}\left\|(-A)^{\beta/2}S_2(t-s)F(X(s))\right\|^2_{L^2(\Omega,H)}ds\right)^2\nonumber\\
			&+&4\int_{0}^{t}(t-s)^{2\alpha-2}\mathbb{E}\left(\left\|(-A)^{\beta/2}S_2(t-s)B(X(s))\right\|_{L^0_2}^2\right)ds\nonumber\\
			&+&4\int_{0}^{t}\int_{\mathcal{X}}(t-s)^{2\alpha-2}\left\|(-A)^{\beta/2}S_2(t-s)G(z,X(s))\right\|_{L^2(\Omega,H)}^2\upsilon(dz)ds\nonumber\\
			&=:&4\sum_{i=1}^4 I_i^2.
		\end{eqnarray}
		We will bound $I_i^2$, $i=1,2,3,4$ one by one. First, by the boundedness of $S_1(t)$ in Lemma \ref{bound} and \eqref{semigroup_prp*}, we have
		\begin{eqnarray}
			\label{spareg1}
			I_1^2:=\left\|(-A)^{\beta/2}S_1(t)X_0\right\|^2_{L^2(\Omega,H)}\leq C\left\|(-A)^{\beta/2}X_0\right\|^2_{L^2(\Omega,H)}.
		\end{eqnarray}
		For $I_2$, by semigroup property \eqref{semigroup_prp} with $\rho=\frac{\beta}2$, Assumption \ref{nonlin} and \eqref{boun}, we get
		\begin{eqnarray}
			\label{spareg2}
			I_2^2&:=&\left(\int_{0}^{t}(t-s)^{\alpha-1}\left\|(-A)^{\beta/2}S_2(t-s)F(X(s))\right\|_{L^2(\Omega,H)}ds\right)^2\nonumber\\
			&\leq& C\left(\int_{0}^{t}(t-s)^{\alpha-1}(t-s)^{-\frac{\alpha\beta}2}\left(1+\|X(s)\|_{L^2(\Omega,H)}\right)ds\right)^2\nonumber\\
			&\leq& C\left(\int_{0}^{t}(t-s)^{\alpha-\frac{\alpha\beta}2-1}ds\right)^2\left(1+\mathbb{E}\left[\sup_{0\leq s\leq t}\|X(s)\|^2\right]\right)\nonumber\\
			&\leq& Ct^{\alpha(2-\beta)}\left(1+\mathbb{E}\left[\sup_{0\leq t\leq T}\|X(t)\|^2\right]\right)\leq C.
		\end{eqnarray}
		Applying \eqref{semigroup_prp2}, the boundedness of $S_2(t)$ in Lemma \ref{bound},  Assumption \ref{growth} with $\tau=\frac{\beta}2$, we deduce
		\begin{eqnarray}
			\label{spareg3}
			I_3^2&:=&\int_{0}^{t}(t-s)^{2\alpha-2}\mathbb{E}\left[\left\|(-A)^{\beta/2}S_2(t-s)B(X(s))\right\|_{L^0_2}^2\right]ds\nonumber\\
			&\leq& C\int_{0}^{t}(t-s)^{2\alpha-2}\left\|S_2(t-s)\right\|_{L(H)}^2\left(1+\mathbb{E}\left[\|(-A)^{\beta/2}X(s)\|^2\right]\right)ds\nonumber\\
			&\leq& C \int_{0}^{t}(t-s)^{2\alpha-2}\left(1+\|(-A)^{\beta/2}X(s)\|^2_{L^2(\Omega;H)}\right)ds\nonumber\\
			&\leq& C+C \int_{0}^{t}(t-s)^{2\alpha-2}\|(-A)^{\beta/2}X(s)\|^2_{L^2(\Omega;H)}ds.
		\end{eqnarray}
		For $I_4$, analogous to $I_3$, we use also \eqref{semigroup_prp2}, the boundedness of $S_2(t)$ in Lemma \ref{bound},  Assumption \ref{growth} with $\tau=\frac{\beta}2$, to obtain 
		\begin{eqnarray}
			\label{spareg4}
			I_4^2&:=&\mathbb{E}\left[\int_{0}^{t}\int_{\mathcal{X}}(t-s)^{2\alpha-2}\left\|(-A)^{\beta/2}S_2(t-s)G(z,X(s))\right\|^2\upsilon(dz)ds\right]\nonumber\\
			&=&\mathbb{E}\left[\int_{0}^{t}\int_{\mathcal{X}}(t-s)^{2\alpha-2}\left\|S_2(t-s)\right\|^2_{L(H)}\left\|(-A)^{\beta/2}G(z,X(s))\right\|^2\upsilon(dz)ds\right]\nonumber\\
			&\leq& C \mathbb{E}\left[\int_{0}^{t}(t-s)^{2\alpha-2}\int_{\mathcal{X}}\left\|(-A)^{\beta/2}G(z,X(s))\right\|^2\upsilon(dz)ds\right]\nonumber\\
			&\leq& C \int_{0}^{t}(t-s)^{2\alpha-2}\left(1+\|(-A)^{\beta/2}X(s)\|^2_{L^2(\Omega;H)}\right)ds\nonumber\\
			&\leq& C+C \int_{0}^{t}(t-s)^{2\alpha-2}\|(-A)^{\beta/2}X(s)\|^2_{L^2(\Omega;H)}ds.
		\end{eqnarray}
		Putting \eqref{spareg1} - \eqref{spareg4} in \eqref{spareg} hence yields 
		\begin{eqnarray*}
			\|(-A)^{\beta/2}X(t)\|^2_{L^2(\Omega,H)}&\leq& C(1+\|(-A)^{\beta/2}X_0\|^2_{L^2(\Omega,H)})\nonumber\\
			&+&C \int_{0}^{t}(t-s)^{2\alpha-2}\|(-A)^{\beta/2}X(s)\|^2_{L^2(\Omega;H)}ds,
		\end{eqnarray*}
		applying fractional Gronwall's lemma \cite[Lemma A.2]{Ye,Krub} proves \eqref{spa_reg1}.
	\end{proof} 
	Now by the following theorem, we provide a temporal regularity of the solution process of \eqref{model_timfrac}.
	\begin{thm}
		\label{tim_reg}
		Suppose that Assumptions \ref{init} - \ref{growth} are fulfilled. Then the following estimate holds
		\begin{eqnarray}
			\label{tim_reg1}
			\left\|X(t_2)-X(t_1)\right\|_{L^2(\Omega;H)}\leq C(t_2-t_1)^{\frac{\min(\alpha\beta,2-2\alpha)}2},\,\,\,\,\,\,\,0\leq t_1<t_2\leq T. 
		\end{eqnarray}
	\end{thm}
	Moreover, \eqref{tim_reg1} holds when $A$ and $X$ are replaced by their semidiscrete versions $A_h$ and $X^h$ respectively, defined in Section \ref{spa_conv}.\\
	\begin{proof}
		For $0\leq 
		t_1<t_2\leq T$, we rewrite the mild solution \eqref{sol} at times $t=t_2$ and $t=t_1$ and we subtract $X(t_2)$ by $X(t_1)$ as
		\begin{eqnarray}
			\label{tim_reg*}
			&&X(t_2)-X(t_1)\nonumber\\
			&=&\left(S_1(t_2)-S_1(t_1)\right)X_0\nonumber\\
			&+&\int_{0}^{t_1}\left[(t_2-s)^{\alpha-1}S_2(t_2-s)-(t_1-s)^{\alpha-1}S_2(t_1-s)\right]F(X(s))ds\nonumber\\
			&+&\int_{0}^{t_1}\left[(t_2-s)^{\alpha-1}S_2(t_2-s)-(t_1-s)^{\alpha-1}S_2(t_1-s)\right]B(X(s))dW(s)\nonumber\\
			&+&\int_{0}^{t_1}\int_{\mathcal{X}}\left[(t_2-s)^{\alpha-1}S_2(t_2-s)-(t_1-s)^{\alpha-1}S_2(t_1-s)\right]G(z,X(s))\widetilde{N}(dz,ds)\nonumber\\
			&+&\int_{t_1}^{t_2}(t_2-s)^{\alpha-1}S_2(t_2-s)F(X(s))ds\nonumber\\
			&+&\int_{t_1}^{t_2}(t_2-s)^{\alpha-1}S_2(t_2-s)B(X(s))dW(s)\nonumber\\
			&+&\int_{t_1}^{t_2}\int_{\mathcal{X}}(t_2-s)^{\alpha-1}S_2(t_2-s)G(z,X(s))\widetilde{N}(dz,ds).
		\end{eqnarray}
		Taking the $L^2$ norm in both sides and using triangle inequality yields
		\begin{eqnarray}
			\label{timreg}
			\left\|X(t_2)-X(t_1)\right\|_{L^2(\Omega;H)}\leq \sum_{i=1}^7J_i.
		\end{eqnarray}
		Inserting an appropriate power of $(-A)$, using Lemma \ref{semigroup12} with $\eta=\frac{\beta}2$ and Assumption \ref{init} implies 
		\begin{eqnarray}
			\label{timreg1}
			J_1&:=& \|\left(S_1(t_2)-S_1(t_1)\right)X_0\|_{L^2(\Omega;H)}\nonumber\\
			&=& \|(-A)^{-\frac{\beta}2}\left(S_1(t_2)-S_1(t_1)\right)(-A)^{\frac{\beta}2}X_0\|_{L^2(\Omega;H)}\nonumber\\
			&\leq& C(t_2-t_1)^{\frac{\alpha\beta}2}\|(-A)^{\frac{\beta}2}X_0\|_{L^2(\Omega;H)}\nonumber\\
			&\leq& C(t_2-t_1)^{\frac{\alpha\beta}2}.
		\end{eqnarray}
		By triangle inequality, using \eqref{semigroup_prp1*}, Assumption \ref{nonlin}, \eqref{boun} and Cauchy-Schwartz inequality, we get
		\begin{eqnarray}
			\label{timreg2}
			J_2&:=&\left\|\int_{0}^{t_1}\left[(t_2-s)^{\alpha-1}S_2(t_2-s)-(t_1-s)^{\alpha-1}S_2(t_1-s)\right]F(X(s))ds\right\|_{L^2(\Omega;H)}\nonumber\\
			&\leq& \int_{0}^{t_1}\|\left[(t_2-s)^{\alpha-1}S_2(t_2-s)-(t_1-s)^{\alpha-1}S_2(t_1-s)\right]F(X(s))\|_{L^2(\Omega;H)}ds\nonumber\\
			&\leq& C\int_{t_1}^{t_2}\left\|(t_2-s)^{\alpha-1}S_2(t_2-s)-(t_1-s)^{\alpha-1}S_2(t_1-s)\right\|_{L(H)}\|F(X(s))\|_{L^2(\Omega;H)}ds\nonumber\\
			&\leq& C\left(\int_{0}^{t_1}(t_2-t_1)^{1-\alpha}(t_1-s)^{\alpha-1}(t_2-s)^{\alpha-1}ds\right)\left(1+\mathbb{E}[\sup_{0\leq s\leq T}\|X(s)\|^2]\right)^{\frac 12}\nonumber\\
			&\leq& C(t_2-t_1)^{1-\alpha}\left(\int_{0}^{t_1}(t_2-s)^{2\alpha-2}ds\right)^{\frac 12}\left(\int_{0}^{t_1}(t_1-s)^{2\alpha-2}ds\right)^{\frac 12}\left(1+\mathbb{E}[\sup_{0\leq s\leq T}\|X(s)\|^2]\right)^{\frac 12}\nonumber\\
			&\leq& C\frac{(t_2-t_1)^{1-\alpha}}{2\alpha-1}\left(t_2^{2\alpha-2}-(t_2-t_1)^{2\alpha-1}\right)^{\frac 12}(t_1)^{\frac {2\alpha-2}2}\nonumber\\
			&\leq& C(t_2-t_1)^{1-\alpha}t_1^{2\alpha-1}\leq C(t_2-t_1)^{1-\alpha}.
		\end{eqnarray}
		Using the It\^{o} isometry \eqref{jumpint}, \eqref{semigroup_prp1*}, Assumption \ref{growth} with $r=0$, \eqref{boun} and Cauchy-Schwartz inequality, we obtain
		\begin{eqnarray}
			\label{timreg4}
			J_4^2&=&\left\|\int_{0}^{t_1}\int_{\mathcal{X}}\left[(t_2-s)^{\alpha-1}S_2(t_2-s)-(t_1-s)^{\alpha-1}S_2(t_1-s)\right]G(z,X(s))\widetilde{N}(dz,ds)\right\|^2_{L^2(\Omega;H)}\nonumber\\
			&=&	\mathbb{E}\left[\int_{0}^{t_1}\int_{\mathcal{X}}\left\|\left[(t_2-s)^{\alpha-1}S_2(t_2-s)-(t_1-s)^{\alpha-1}S_2(t_1-s)\right]G(z,X(s))\right\|^2\nu(dz)ds\right]\nonumber\\
			&\leq&	\int_{0}^{t_1}\left\|(t_2-s)^{\alpha-1}S_2(t_2-s)-(t_1-s)^{\alpha-1}S_2(t_1-s)\right\|^2_{L(H)}\mathbb{E}\left[\int_{\mathcal{X}}\left\|G(z,X(s))\right\|^2\nu(dz)\right]ds\nonumber\\
			&\leq& C\left(\int_{0}^{t_1}(t_2-t_1)^{2-2\alpha}(t_1-s)^{2\alpha-2}(t_2-s)^{2\alpha-2}ds\right)\left(1+\mathbb{E}[\sup_{0\leq s\leq T}\|X(s)\|^2]\right)\nonumber\\
			&\leq& C(t_2-t_1)^{2-2\alpha}\left(\int_{0}^{t_1}(t_1-s)^{2\alpha-2}(t_2-s)^{2\alpha-2}ds\right)\nonumber\\
			&\leq& C(t_2-t_1)^{2-2\alpha}\left(\int_{0}^{t_1}(t_2-s)^{4\alpha-4}ds\right)^{\frac 12}\left(\int_{0}^{t_1}(t_1-s)^{4\alpha-4}ds\right)^{\frac 12}\nonumber\\
			&\leq& C\frac{(t_2-t_1)^{2-2\alpha}}{4\alpha-3}\left(t_2^{4\alpha-3}-(t_2-t_1)^{4\alpha-3}\right)^{\frac 12}(t_1)^{\frac {4\alpha-3}2}\nonumber\\
			&\leq& C(t_2-t_1)^{2-2\alpha}t_1^{4\alpha-3}\leq C(t_2-t_1)^{2-2\alpha}.
		\end{eqnarray}
		By using a similar procedure as bounding $J_4$, using the It\^{o} isometry \eqref{browint}, \eqref{semigroup_prp1*}, Assumption \ref{growth} with $\tau=0$, \eqref{boun} and Cauchy-Schwartz inequality, we have
		\begin{eqnarray}
			\label{timreg3}
			J_3^2&\leq&\left\|\int_0^{t_1}\left[(t_2-s)^{\alpha-1}S_2(t_2-s)-(t_1-s)^{\alpha-1}S_2(t_1-s)\right]B(X(s))dW(s)\right\|^2_{L^2(\Omega;H)}\nonumber\\\nonumber\\
			&\leq& C(t_2-t_1)^{2-2\alpha}.
		\end{eqnarray}
		By triangle inequality, the boundedness of operator $S_2(t)$ \eqref{prop_semi}, Assumption \ref{nonlin} and \eqref{boun}, we get
		
		\begin{eqnarray}
			\label{timreg5}
			J_5&:=&\left\|\int_{t_1}^{t_2}(t_2-s)^{\alpha-1}S_2(t_2-s)F(X(s))ds\right\|_{L^2(\Omega;H)}\nonumber\\
			&\leq& \int_{t_1}^{t_2}(t_2-s)^{\alpha-1}\|S_2(t_2-s)F(X(s))\|_{L^2(\Omega;H)}ds\nonumber\\
			&\leq& C\int_{t_1}^{t_2}(t_2-s)^{\alpha-1}\|F(X(s))\|_{L^2(\Omega;H)}ds\nonumber\\
			&\leq& C\left(\int_{t_1}^{t_2}(t_2-s)^{\alpha-1}ds\right)\left(1+\mathbb{E}[\sup_{0\leq s\leq T}\|X(s)\|^2]\right)^{\frac 12}\nonumber\\
			&\leq& C(t_2-t_1)^{\alpha}.
		\end{eqnarray}
		Now to bound the sixth term, the It\^o isometry property \eqref{browint} leads
		\begin{eqnarray*}
			J_6^2&:=&\left\|\int_{t_1}^{t_2}(t_2-s)^{\alpha-1}S_2(t_2-s)B(X(s))dW(s)\right\|^2_{L^2(\Omega;H)}\nonumber\\
			&=&\mathbb{E}\left[\int_{t_1}^{t_2}\left\|(t_2-s)^{\alpha-1}S_2(t_2-s)B(X(s))\right\|^2_{L^0_2}ds\right].
		\end{eqnarray*}
		By the boundedness of operator $S_2(t)$ \eqref{prop_semi}, Assumption \ref{growth} with $\tau=0$ and \eqref{boun} we have
		\begin{eqnarray}
			\label{timreg6}
			J_6^2&\leq&\int_{t_1}^{t_2}(t_2-s)^{2\alpha-2}\left\|S_2(t_2-s)\right\|^2_{L(H)}\mathbb{E}\|B(X(s))\|^2_{L^0_2}ds\nonumber\\
			&\leq& C\left(\int_{t_1}^{t_2}(t_2-s)^{2\alpha-2}ds\right)\left(1+\mathbb{E}[\sup_{0\leq s\leq T}\|X(s)\|^2]\right)\nonumber\\
			&\leq& C(t_2-t_1)^{2\alpha-1}.
		\end{eqnarray}
		By using a similar procedure as bounding $J_6$, using the It\^{o} isometry \eqref{jumpint}, the boundedness of operator $S_2(t)$ \eqref{prop_semi}, Assumption \ref{growth} with $r=0$ and \eqref{boun}, we easily have
		\begin{eqnarray}
			\label{timreg7}
			J_7^2&=&\left\|\int_{t_1}^{t_2}\int_{\mathcal{X}}(t_2-s)^{\alpha-1}S_2(t_2-s)G(z,X(s))\widetilde{N}(dz,ds)\right\|^2_{L^2(\Omega;H)}\nonumber\\
			&\leq& C(t_2-t_1)^{2\alpha-1}.
		\end{eqnarray}
		Substituting \eqref{timreg1} - \eqref{timreg7} in \eqref{timreg} and note that for $\alpha\in\left(\frac 34,1\right)$, $2\alpha-1>2-2\alpha$ then
		\begin{eqnarray}
			\label{timreg*}
			\left\|X(t_2)-X(t_1)\right\|_{L^2(\Omega;H)}\leq C(t_2-t_1)^{\frac{\min(\alpha\beta,2-2\alpha)}2},\,\,\,\,\,\,\,0<t_1<t_2\leq T. 
		\end{eqnarray}
		The proof of Theorem \ref{tim_reg} is thus completed. 
	\end{proof}
	\section{Space approximation and error estimates}
	\label{spa_conv}
	We consider the discretization of the spatial domain by a finite element triangulation with maximal length $h$ satisfying the usual regularity assumptions. Let $V_h\subset V$ denotes the space of continuous functions that are piecewise linear over triangulation $J_h$. To discretise in space, we introduce $P_h$ from $L^2(\Omega)$ to $V_h$ define for $u\in L^2(\Omega)$ by
	\begin{eqnarray}
		\label{pro}
		(P_h u,\xi)=(u,\xi),\hspace{2cm}\forall \xi\in V_h.
	\end{eqnarray}
	The discrete operator $A_h: V_h\rightarrow V_h$ is defined by
	\begin{eqnarray}
		\label{dis}
		(A_h \rho,\xi)=-a(\rho,\xi),\hspace{2cm}\forall \rho, \xi\in V_h,
	\end{eqnarray}
	where $a$ is the corresponding bilinear form of $A$. Like the operator $A$, the discrete operator $A_h$ is also the generator of an analytic semigroup $S_h(t):=e^{tA_h}$. The semidiscrete space version of problem \eqref{model_timfrac} is to find $X^h(t)=X^h(\cdot,t)$ such that for $t\in[0,T]$
	\begin{eqnarray}
		\label{modeldis_timfrac}
		\left\{
		\begin{array}{ll}
			\partial^{\alpha}_tX^h(t)=A_hX^h(t)+P_hF(X^h(t))+P_hB(X^h(t))\frac{dW(t)}{dt}+\frac{\int_{\mathcal{X}}P_hG(z,X^h(t))\widetilde{N}(dz,dt)}{dt}, \\
			X^h(0)=P_hX_0,\,\,\,\,t\in[0,T].
		\end{array}
		\right.
	\end{eqnarray}
	Note that $A_h$, $P_hF$, $P_hB$ and $P_hG$ satisfy the same assumptions as $A$, $F$, $B$ and $G$ respectively. The mild solution of \eqref{modeldis_timfrac} can be represented as follows
	\begin{eqnarray}
		\label{dissol}
		X^h(t)&=&S_{1h}(t)X^h_0+\int_{0}^{t}(t-s)^{\alpha-1}S_{2h}(t-s)P_hF(X^h(s))ds\nonumber\\
		&+&\int_{0}^{t}(t-s)^{\alpha-1}S_{2h}(t-s)P_hB(X^h(s))dW(s)\nonumber\\
		&+&\int_{0}^{t}\int_{\mathcal{X}}(t-s)^{\alpha-1}S_{2h}(t-s)P_hG(z,X^h(s))\widetilde{N}(dz,ds),
	\end{eqnarray}
	where $S_{1h}$ and $S_{2h}$ are the semi discrete version of $S_1$ and $S_2$ respectively defined by \eqref{S1} and \eqref{S2}.
	Let us define the error operators 
	\begin{eqnarray*}
		T_h(t):=S(t)-S_h(t)P_h,\quad T_{1h}(t):=S(t)-S_{1h}(t)P_h,\quad T_{2h}(t):=S(t)-S_{2h}(t)P_h.
	\end{eqnarray*}
	Then we have the following lemma.
	\begin{lem}
		\label{disop}
		\begin{enumerate}
			\item [(i)] Let $r\in[0,2]$, $\rho\leq r$, $t\in(0,T]$, $v\in D((-A)^{\rho})$. Then there exists a positive constant C such that
			\begin{eqnarray}
				\label{nordisop}
				\|T_h(t)v\|\leq C h^rt^{-(r-\rho)/2}\left\|v\right\|_{\rho}, 
			\end{eqnarray}
			and
			\begin{eqnarray}
				\label{fradisop1}
				\|T_{1h}(t)v\|\leq C h^rt^{-\alpha(r-\rho)/2}\left\|v\right\|_{\rho},\quad\|T_{2h}(t)v\|\leq C h^rt^{-\alpha(r-\rho)/2}\left\|v\right\|_{\rho}.
			\end{eqnarray}
			\item [(ii)] Let $0\leq\gamma\leq 1$, then there exists a constant $C$ such that 
			\begin{eqnarray}
				\label{fradisop2}
				\left\|\int_0^ts^{\alpha-1}T_{2h}(s)v ds\right\|\leq C h^{2-\gamma} \left\|v\right\|_{-\gamma},\quad v\in D((-A)^{-\gamma}),\,\,t>0.
			\end{eqnarray}
		\end{enumerate}
	\end{lem}
	\begin{proof}
		\begin{enumerate}
			\item [(i)] See \cite[Lemma 3.1]{manto}, for the proof of \eqref{nordisop}. For the proof of \eqref{fradisop1}, using \eqref{S1}, \eqref{S2}, their semi discrete forms, \eqref{nordisop}, \eqref{Mai1} and \eqref{Mai2}, we get
			\begin{eqnarray*}
				\|T_{1h}(t)v\|&=&\|(S_1(t)-S_{1h}(t)P_h)v\|\\
				&=&\left\|\int_0^{\infty}M_{\alpha}(\theta)(S(\theta t^{\alpha})-S_h(\theta t^{\alpha})P_h)vd\theta\right\|\\
				&\leq& \int_0^{\infty}M_{\alpha}(\theta)\|T_h(\theta t^{\alpha})v\|d\theta\\
				&\leq& C\int_0^{\infty}M_{\alpha}(\theta)h^r(\theta t^{\alpha})^{-(r-\rho)/2}\|v\|_{\rho}d\theta\\
				&\leq& Ch^rt^{-\alpha(r-\rho)/2}\int_0^{\infty}M_{\alpha}(\theta)\theta^{-(r-\rho)/2}d\theta\|v\|_{\rho}\\
				&\leq& C\frac{\Gamma\left(1-\frac{r-\rho}2\right)}{\Gamma\left(1-\frac{\alpha(r-\rho)}2\right)}h^rt^{-\alpha(r-\rho)/2}\|v\|_{\rho}\\
				&\leq& Ch^rt^{-\alpha(r-\rho)/2}\|v\|_{\rho},
			\end{eqnarray*}
			and
			\begin{eqnarray*}
				\|T_{2h}(t)v\|&=&\|(S_2(t)-S_{2h}(t)P_h)v\|\\
				&=&\left\|\int_0^{\infty}\alpha\theta M_{\alpha}(\theta)(S(\theta t^{\alpha})-S_h(\theta t^{\alpha})P_h)vd\theta\right\|\\
				&\leq& \int_0^{\infty}\alpha\theta M_{\alpha}(\theta)\|T_h(\theta t^{\alpha})v\|d\theta\\
				&\leq& C\int_0^{\infty}\alpha\theta M_{\alpha}(\theta)h^r(\theta t^{\alpha})^{-(r-\rho)/2}\|v\|_{\rho}d\theta\\
				&\leq& C\alpha h^rt^{-\alpha(r-\rho)/2}\int_0^{\infty}M_{\alpha}(\theta)\theta^{1-(r-\rho)/2}d\theta\|v\|_{\rho}\\
				&\leq& C\frac{\Gamma\left(2-\frac{r-\rho}2\right)}{\Gamma\left(1+\alpha \left(1-\frac{r-\rho}2\right)\right)}h^rt^{-\alpha(r-\rho)/2}\|v\|_{\rho}\\
				&\leq& Ch^rt^{-\alpha(r-\rho)/2}\|v\|_{\rho}.
			\end{eqnarray*}
			\item [(ii)] Using \eqref{S2} and its semidiscrete form, \eqref{Mai1}, we obtain
			\begin{eqnarray*}
				&&\int_0^ts^{\alpha-1}T_{2h}(s)v ds\nonumber\\
				&=&\int_0^ts^{\alpha-1}S_2(s)v ds-\int_0^ts^{\alpha-1}S_{2h}(s)P_hv ds\\
				&=&\int_0^ts^{\alpha-1}\int_{0}^{\infty}\alpha\theta M_{\alpha}(\theta)S(\theta s^\alpha)v d\theta ds-\int_0^ts^{\alpha-1}\int_{0}^{\infty}\alpha\theta M_{\alpha}(\theta)S_h(\theta s^\alpha)P_hv d\theta ds\\
				&=&\int_{0}^{\infty}A^{-1} M_{\alpha}(\theta)\int_0^t A\alpha\theta s^{\alpha-1}S(\theta s^\alpha)v d\theta ds-\int_{0}^{\infty}A_h^{-1} M_{\alpha}(\theta)\int_0^t A_h\alpha\theta s^{\alpha-1}S_h(\theta s^\alpha)P_hv d\theta ds\\
				&=&\int_{0}^{\infty}A^{-1} M_{\alpha}(\theta)\left[S(\theta s^{\alpha})\right]^{t}_0 v d\theta ds-\int_{0}^{\infty}A_h^{-1} M_{\alpha}(\theta)\left[S_h(\theta s^{\alpha})\right]^{t}_0P_hv d\theta ds\\
				&=&\int_{0}^{\infty}A^{-1} M_{\alpha}(\theta)\left(S(\theta s^{\alpha})-I\right)v d\theta ds-\int_{0}^{\infty}A_h^{-1} M_{\alpha}(\theta)\left(S_h(\theta s^{\alpha})-I\right)P_hv d\theta ds\\
				&=&\int_{0}^{\infty} M_{\alpha}(\theta)(A_h^{-1}P_h-A^{-1})vd\theta+\int_{0}^{\infty} M_{\alpha}(\theta)(A^{-1}S(\theta t^{\alpha})-A_h^{-1}S_h(\theta t^{\alpha})P_h)vd\theta\\
				&=&(A_h^{-1}P_h-A^{-1})v+\int_{0}^{\infty} M_{\alpha}(\theta)(A^{-1}S(\theta t^{\alpha})-A_h^{-1}S_h(\theta t^{\alpha})P_h)vd\theta,
			\end{eqnarray*}
			and \cite[(65), (69)]{manto} allow to have
			\begin{eqnarray}
				&&\left\|\int_0^ts^{\alpha-1}T_{2h}(s)v ds\right\|\nonumber\\
				&\leq& \left\|(A_h^{-1}P_h-A^{-1})v\right\|+\left\|\int_{0}^{\infty} M_{\alpha}(\theta)(A^{-1}S(\theta t^{\alpha})-A_h^{-1}S_h(\theta t^{\alpha})P_h)vd\theta\right\|\nonumber\\
				&\leq& \left\|(A_h^{-1}P_h-A^{-1})v\right\|+\int_{0}^{\infty} M_{\alpha}(\theta)\left\|(A^{-1}S(\theta t^{\alpha})-A_h^{-1}S_h(\theta t^{\alpha})P_h)v\right\|d\theta\nonumber\\
				&\leq& C h^{2-\gamma} \left\|v\right\|_{-\gamma}+C \int_{0}^{\infty} M_{\alpha}(\theta)h^{2-\gamma} \left\|v\right\|_{-\gamma}d\theta\nonumber\\
				&\leq& C h^{2-\gamma} \left\|v\right\|_{-\gamma}+C h^{2-\gamma} \left\|v\right\|_{-\gamma}\int_{0}^{\infty} M_{\alpha}(\theta)d\theta\nonumber\\
				&\leq& C h^{2-\gamma} \left\|v\right\|_{-\gamma}.
			\end{eqnarray}
		\end{enumerate}
		This completes the proof of Lemma \ref{disop}.
	\end{proof}
	
	We are now in position to prove one of our  main results, which provides an estimate in mean square sense of the error between the solution of SPDE \eqref{model_timfrac} and the spatially semidiscrete approximation \eqref{modeldis_timfrac}.
	\begin{thm}[Space error]
		\label{spaconv}
		Let $X$ and $X^h$ be the mild solution of \eqref{model_timfrac} and \eqref{modeldis_timfrac}, respectively. Suppose that Assumptions \ref{init} - \ref{growth} hold. We have the following estimates depending on the regularity  parameter   $\beta$ of the initial solution $X_0$, 
		\begin{eqnarray}
			\label{spaconv1}
			\left\|X(t)-X^h(t)\right\|_{L^2(\Omega;H)}\leq C h^{\beta},\hspace{1cm}0\leq t\leq T. 
		\end{eqnarray}
	\end{thm}
	\begin{proof}
		Define $e(t):=X(t)-X^h(t)$. By \eqref{sol} and \eqref{dissol}, we deduce
		\begin{eqnarray*}
			e(t)&=&S_1(t)X_0-S_{1h}(t)P_hX_0\\
			&+&\int_{0}^{t}(t-s)^{\alpha-1}S_2(t-s)F(X(s))ds-\int_{0}^{t}(t-s)^{\alpha-1}S_{2	h}(t-s)P_hF(X^h(s))ds\\
			&+&\int_{0}^{t}(t-s)^{\alpha-1}S_2(t-s)B(X(s))dW(s)-\int_{0}^{t}(t-s)^{\alpha-1}S_{2h}(t-s)P_hB(X^h(s))dW(s)\\
			&+&\int_{0}^{t}\int_{\mathcal{X}}(t-s)^{\alpha-1}S_2(t-s)G(z,X(s))\widetilde{N}(dz,ds)\\
			&-&\int_{0}^{t}\int_{\mathcal{X}}(t-s)^{\alpha-1}S_{2h}(t-s)P_hG(z,X(s))\widetilde{N}(dz,ds)\\
			&=:& I+II+III+IV.
		\end{eqnarray*}
		Thus taking the $L^2$ norm and using triangle inequality we have
		\begin{eqnarray}
			\label{spaerr}
			\left\|e(t)\right\|_{L^2(\Omega;H)}&\leq& \left\|I\right\|_{L^2(\Omega;H)}+\left\|II\right\|_{L^2(\Omega;H)}+\left\|III\right\|_{L^2(\Omega;H)}\nonumber\\
			&+&\left\|IV\right\|_{L^2(\Omega;H)}.
		\end{eqnarray}
		We will bound the above terms one by one. For the first term $\left\|I\right\|_{L^2(\Omega;H)}$, using \eqref{fradisop1} with $r=\rho=\beta$ and Assumption \ref{init} yields
		\begin{eqnarray}
			\label{spaerr1}
			\left\|I\right\|_{L^2(\Omega;H)}&:=&\left\|S_1(t)X_0-S_{1h}(t)P_hX_0\right\|_{L^2(\Omega;H)}\nonumber\\
			&=&\left\|T_{1h}(t)X_0\right\|_{L^2(\Omega;H)}\nonumber\\
			&\leq& C h^{\beta}\left\|(-A)^{\beta/2}X_0\right\|_{L^2(\Omega;H)}\leq C h^{\beta}.
		\end{eqnarray}
		For the second term $\left\|II\right\|_{L^2(\Omega;H)}$, adding and substracting a term, using triangle inequality gives
		\begin{eqnarray}
			\label{spaerr2*}
			\left\|II\right\|_{L^2(\Omega;H)}&=&\left\|\int_{0}^{t}(t-s)^{\alpha-1}S_2(t-s)F(X(s))ds\right.\nonumber\\
			&&\hspace{2cm}\left.-\int_{0}^{t}(t-s)^{\alpha-1}S_{2	h}(t-s)P_hF(X^h(s))ds\right\|_{L^2(\Omega;H)}\nonumber\\
			&\leq& \left\|\int_{0}^{t}(t-s)^{\alpha-1}(S_2(t-s)-S_{2	h}(t-s)P_h)F(X(s))ds\right\|_{L^2(\Omega;H)}\nonumber\\
			&+&\left\|\int_{0}^{t}(t-s)^{\alpha-1}S_{2	h}(t-s)P_h(F(X(s))-F(X^h(s)))ds\right\|_{L^2(\Omega;H)}\nonumber\\
			&=:&\left\|II_1\right\|_{L^2(\Omega;H)}+\left\|II_2\right\|_{L^2(\Omega;H)}.
		\end{eqnarray}
		To estimate the term $\left\|II_1\right\|_{L^2(\Omega;H)}$, we also add and subtract a term.  Using  the  triangle inequality yields
		\begin{eqnarray}
			\label{spaerr21}
			\left\|II_1\right\|_{L^2(\Omega;H)}&:=&\left\|\int_{0}^{t}(t-s)^{\alpha-1}(S_2(t-s)-S_{2	h}(t-s)P_h)F(X(s))ds\right\|_{L^2(\Omega;H)}\nonumber\\
			&\leq& \left\|\int_{0}^{t}(t-s)^{\alpha-1}(S_2(t-s)-S_{2	h}(t-s)P_h)(F(X(s))-F(X(t)))ds\right\|_{L^2(\Omega;H)}\nonumber\\
			&+&\left\|\int_{0}^{t}(t-s)^{\alpha-1}(S_2(t-s)-S_{2h}(t-s)P_h)F(X(t))ds\right\|_{L^2(\Omega;H)}\nonumber\\
			&=:&\left\|II_{11}\right\|_{L^2(\Omega;H)}+\left\|II_{12}\right\|_{L^2(\Omega;H)}.
		\end{eqnarray}
		We estimate these two terms separately. Using Cauchy-Schwartz inequality, \eqref{fradisop1} with $r=\beta$, $\rho=0$, Assumption \ref{nonlin} and Theorem \ref{tim_reg}, leads to
		\begin{eqnarray}
			\label{spaerr211}
			\left\|II_{11}\right\|_{L^2(\Omega;H)}&:=&\left\|\int_{0}^{t}(t-s)^{\alpha-1}T_{2	h}(t-s)(F(X(s))-F(X(t)))ds\right\|_{L^2(\Omega;H)}\nonumber\\
			&\leq& \int_{0}^{t}(t-s)^{\alpha-1}\left\|T_{2	h}(t-s)(F(X(s))-F(X(t)))\right\|_{L^2(\Omega;H)}ds\nonumber\\
			&\leq& C h^{\beta}\int_{0}^{t}(t-s)^{\alpha-1}(t-s)^{-\frac{\alpha\beta}2}(t-s)^{\frac{\min(\alpha\beta,2\alpha-1)}2}ds\nonumber\\
			&\leq& C h^{\beta}\int_{0}^{t}(t-s)^{\min\left(\alpha-1,\frac{\alpha(4-\beta)-3}2\right)}ds\nonumber\\
			&\leq& C h^{\beta}t^{\min\left(\alpha,\frac{\alpha(4-\beta)-1}2\right)}\leq C h^{\beta}.
		\end{eqnarray}
		As with the term $\left\|II_{11}\right\|_{L^2(\Omega;H)}$, by applying \eqref{fradisop2} with $\gamma=0$, Assumption \ref{nonlin} and Theorem \ref{well_poss}, we get
		\begin{eqnarray}
			\label{spaerr212}
			\left\|II_{12}\right\|_{L^2(\Omega;H)}&=&\left\|\int_{0}^{t}(t-s)^{\alpha-1}T_{2h}(t-s)F(X(t))ds\right\|_{L^2(\Omega;H)}\nonumber\\
			&\leq& Ch^2\left\|F(X(t))\right\|_{L^2(\Omega;H)}\nonumber\\
			&\leq& C h^2\left(1+\mathbb{E}\left[\sup_{0\leq t\leq T}\left\|X(t)\right\|^2\right]\right)^{\frac 12}\leq C h^2. 
		\end{eqnarray}
		For the term $\left\|II_2\right\|_{L^2(\Omega;H)}$, using Cauchy-Schwartz inequality and the fact that $S_{2h}(t-s)$ and $P_h$ are bounded, Assumption \ref{nonlin}, we have
		\begin{eqnarray}
			\label{spaerr22}
			\left\|II_2\right\|^2_{L^2(\Omega;H)}&:=&\left\|\int_{0}^{t}(t-s)^{\alpha-1}S_{2	h}(t-s)P_h(F(X(s))-F(X^h(s)))ds\right\|^2_{L^2(\Omega;H)}\nonumber\\
			&\leq& C\int_{0}^{t}(t-s)^{2\alpha-2}\left\|S_{2	h}(t-s)P_h\right\|^2_{L(H)}\left\|F(X(s))-F(X^h(s))\right\|^2_{L^2(\Omega;H)} ds\nonumber\\
			&\leq& C\int_{0}^{t}(t-s)^{2\alpha-2}\left\|e(s)\right\|^2_{L^2(\Omega;H)} ds.
		\end{eqnarray}
		Substituting \eqref{spaerr211}, \eqref{spaerr212} in \eqref{spaerr21}, hence putting \eqref{spaerr21} and \eqref{spaerr22} in \eqref{spaerr2*} gives
		\begin{eqnarray}
			\label{spaerr2}
			\left\|II\right\|^2_{L^2(\Omega;H)}&\leq& Ch^{2\beta}+C h^4+C\int_{0}^{t}(t-s)^{2\alpha-2}\left\|e(s)\right\|^2_{L^2(\Omega;H)} ds\nonumber\\
			&\leq& Ch^{2\beta}+C\int_{0}^{t}(t-s)^{2\alpha-2}\left\|e(s)\right\|^2_{L^2(\Omega;H)} ds.
		\end{eqnarray}
		For the fourth term $\left\|IV\right\|_{L^2(\Omega;H)}$, by adding and subtracting a term, the use of  the triangle inequality yields
		\begin{eqnarray*}
			\left\|IV\right\|_{L^2(\Omega;H)}&=&\left\|\int_{0}^{t}\int_{\mathcal{X}}(t-s)^{\alpha-1}S_2(t-s)G(z,X(s))\widetilde{N}(dz,ds)\right.\nonumber\\
			&-&\left.\int_{0}^{t}\int_{\mathcal{X}}(t-s)^{\alpha-1}S_{2h}(t-s)P_hG(z,X(s))\widetilde{N}(dz,ds)\right\|_{L^2(\Omega;H)}\nonumber\\
			&\leq& \left\|\int_{0}^{t}\int_{\mathcal{X}}(t-s)^{\alpha-1}\left(S_2(t-s)-S_{2h}(t-s)P_h\right)G(z,X(s))\widetilde{N}(dz,ds)\right\|_{L^2(\Omega;H)}\nonumber\\
			&+&\left\|\int_{0}^{t}\int_{\mathcal{X}}(t-s)^{\alpha-1}S_{2h}(t-s)P_h\left(G(z,X(s))-G(z,X^h(s)\right))\widetilde{N}(dz,ds)\right\|_{L^2(\Omega;H)}\nonumber\\
			&=:&\left\|IV_1\right\|_{L^2(\Omega;H)}+\left\|IV_2\right\|_{L^2(\Omega;H)}.
		\end{eqnarray*}
		In a similar way as for $\left\|II\right\|^2_{L^2(\Omega;H)}$, using It\^o isometry \eqref{jumpint}, the boundedness of the operators $S_{2h}(t-s)$ and $P_h$, Assumption \ref{lips}, we deduce
		\begin{eqnarray}
			\label{spaerr42}
			\left\|IV_2\right\|^2_{L^2(\Omega;H)}&=&\left\|\int_{0}^{t}\int_{\mathcal{X}}(t-s)^{\alpha-1}S_{2h}(t-s)P_h\left(G(z,X(s))-G(z,X^h(s)\right))\widetilde{N}(dz,ds)\right\|^2_{L^2(\Omega;H)}\nonumber\\
			&=&\mathbb{E}\left[\int_{0}^{t}\int_{\mathcal{X}}(t-s)^{2\alpha-2}\left\|S_{2h}(t-s)P_h\right\|^2_{L(H)}\left\|G(z,X(s))-G(z,X^h(s))\right\|^2\upsilon(dz)ds\right]\nonumber\\
			&\leq& C\mathbb{E}\left[\int_{0}^{t}(t-s)^{2\alpha-2}\left(\int_{\mathcal{X}}\left\|G(z,X(s))-G(z,X^h(s))\right\|^2\upsilon(dz)\right)ds\right]\nonumber\\
			&\leq& C\mathbb{E}\left[\int_{0}^{t}(t-s)^{2\alpha-2}\left\|X(s)-X^h(s)\right\|^2ds\right]\nonumber\\
			&\leq& C\int_{0}^{t}(t-s)^{2\alpha-2}\left\|e(s)\right\|^2_{L^2(\Omega;H)}ds.
		\end{eqnarray}
		For the estimate $\left\|IV_1\right\|^2_{L^2(\Omega;H)}$, using It\^o isometry \eqref{jumpint}, \eqref{fradisop1} with $r=\rho=\beta$, Assumption \ref{growth} with $\tau=\frac{\beta}2$ and Lemma \ref{spa_reg}, leads to
		\begin{eqnarray}
			\label{spaerr41}
			\left\|IV_1\right\|^2_{L^2(\Omega;H)}&:=&\left\|\int_{0}^{t}\int_{\mathcal{X}}(t-s)^{\alpha-1}T_{2h}(t-s)G(z,X(s))\tilde{N}(dz,ds)\right\|^2_{L^2(\Omega;H)}\nonumber\\
			&=&\mathbb{E}\left[\int_{0}^{t}\int_{\mathcal{X}}(t-s)^{2\alpha-2}\left\|T_{2h}(t-s)G(z,X(s))\right\|^2\upsilon(dz)ds\right]\nonumber\\
			&\leq& C h^{2\beta}\mathbb{E}\left[\int_{0}^{t}(t-s)^{2\alpha-2}\left(\int_{\mathcal{X}}\left\|(-A)^{\frac{\beta}2}G(z,X(s))\right\|^2\upsilon(dz)\right)ds\right]\nonumber\\
			&\leq& C h^{2\beta}\int_{0}^{t}(t-s)^{2\alpha-2} \left(1+\|(-A)^{\frac{\beta}2}X(s)\|^2_{L^2(\Omega,H)}\right)ds\nonumber\\
			&\leq& C h^{2\beta}t^{2\alpha-1}\left(1+\|(-A)^{\frac{\beta}2}X_0\|^2_{L^2(\Omega,H)}\right)\leq  C h^{2\beta}.
		\end{eqnarray}
		Combining \eqref{spaerr42} and \eqref{spaerr41} it results that
		\begin{eqnarray}
			\label{spaerr4*}
			\left\|IV\right\|^2_{L^2(\Omega;H)}\leq  C h^{2\beta}+C\int_{0}^{t}(t-s)^{2\alpha-2}\left\|e(s)\right\|^2_{L^2(\Omega;H)}ds.
		\end{eqnarray}
		Using the similar procedure as in  $\left\|IV_1\right\|^2_{L^2(\Omega;H)}$, we have the following  estimate
		\begin{eqnarray}
			\label{spaerr3*}
			\left\|III\right\|^2_{L^2(\Omega;H)}\leq  C h^{2\beta}+C\int_{0}^{t}(t-s)^{2\alpha-2}\left\|e(s)\right\|^2_{L^2(\Omega;H)}ds.
		\end{eqnarray}
		Combining \eqref{spaerr1}, \eqref{spaerr2}, \eqref{spaerr4*}, \eqref{spaerr3*} and applying the fractional Gronwall's lemma (see \cite{Ye,Krub}) completes the proof of Theorem \ref{spaconv}. 
	\end{proof}
	\section{Fully discrete Euler scheme and its error estimates}
	\label{schemes}
	In this section, we consider a fully discrete approximation of SPDE \eqref{modeldis_timfrac}. Before defining our numerical approximation of the mild solution of the semidiscrete problem \eqref{modeldis_timfrac}, we present a useful to well rewrite this numerical approximation.
	
	We recall that the mild solution at $t_m=m\Delta t$, $\Delta t>0$ of the semi discrete problem \eqref{modeldis_timfrac} is given by
	\begin{eqnarray*}
		X^h(t_m)&=&S_{1h}(t_m)X^h_0+\int_0^{t_m}(t_m-s)^{\alpha-1}S_{2h}(t_m-s)P_hF(X^h(s))ds\\
		&+&\int_0^{t_m}(t_m-s)^{\alpha-1}S_{2h}(t_m-s)P_hB(X^h(s))dW(s)\\
		&+&\int_0^{t_m}\int_{\mathcal{X}}(t_m-s)^{\alpha-1}S_{2h}(t_m-s)P_hG(z,X^h(s))\widetilde{N}(dz,ds).
	\end{eqnarray*}
	
	By decomposing the integrals of the right-hand side of the previous equality using the Chasles relation, we obtain
	
	\begin{eqnarray}
		\label{dissol1}
		&&X^h(t_m)\nonumber\\
		&=&S_{1h}(t_m)X^h_0+\sum_{j=0}^{m-1}\int_{t_j}^{t_{j+1}}(t_m-s)^{\alpha-1}S_{2h}(t_m-s)P_hF(X^h(s))ds\nonumber\\
		&+&\sum_{j=0}^{m-1}\int_{t_j}^{t_{j+1}}(t_m-s)^{\alpha-1}S_{2h}(t_m-s)P_hB(X^h(s))dW(s)\nonumber\\
		&+&\sum_{j=0}^{m-1}\int_{t_j}^{t_{j+1}}\int_{\mathcal{X}}(t_m-s)^{\alpha-1}S_{2h}(t_m-s)P_hG(z,X^h(s))\widetilde{N}(dz,ds).
	\end{eqnarray}
	To build our numerical scheme we use the following approximations for all $s\in[t_j,t_{j+1})$ with $j\in\{0,\,1,\,...\,m-1\}$
	\begin{eqnarray*}
		(t_m-s)^{\alpha-1}S_{2h}(t_{m+1}-s)P_hF(X^h(s))\approx (t_m-t_j)^{\alpha-1}S_{2h}(t_m-t_j)P_hF(X^h(t_j)),
	\end{eqnarray*}
	\begin{eqnarray*} 
		(t_{m+1}-s)^{\alpha-1}S_{2h}(t_{m+1}-s)P_hB(X^h(s))\approx (t_m-t_j)^{\alpha-1}S_{2h}(t_m-t_j)P_hB(X^h(t_j)),
	\end{eqnarray*} 
	\begin{eqnarray*}
		(t_{m+1}-s)^{\alpha-1}S_{2h}(t_{m+1}-s)P_hG(z,X^h(s))\approx (t_m-t_j)^{\alpha-1}S_{2h}(t_m-t_j)P_hG(z,X^h(t_j)),
	\end{eqnarray*} 
	We can define our approximation $X^h_m$ of $X^h(m\Delta t)$ by
	\begin{eqnarray}
		\label{numsol}
		&&X^h_m\nonumber\\
		&=&S_{1h}(t_m)X^h_0+\sum_{j=0}^{m-1}\int_{t_j}^{t_{j+1}}(t_m-t_j)^{\alpha-1}S_{2h}(t_m-t_j)P_hF(X^h_j)ds\nonumber\\
		&+&\sum_{j=0}^{m-1}\int_{t_j}^{t_{j+1}}(t_m-t_j)^{\alpha-1}S_{2h}(t_m-t_j)P_hB(X^h_j)dW(s)\nonumber\\
		&+&\sum_{j=0}^{m-1}\int_{t_j}^{t_{j+1}}\int_{\mathcal{X}}(t_m-t_j)^{\alpha-1}S_{2h}(t_m-t_j)P_hG(z,X^h_j)\widetilde{N}(dz,ds).
	\end{eqnarray}
	Hence using \eqref{S1} and \eqref{S2}, it holds that
	\begin{eqnarray}
		&&X^h_m\nonumber\\
		&=&E_{\alpha,1}(A_ht_m^{\alpha})X^h_{m}+\Delta t\sum_{j=0}^{m-1}(t_m-t_j)^{\alpha-1}E_{\alpha,\alpha}(A_h(t_m-t_j)^{\alpha})S_{2h}P_hF(X^h_j)\nonumber\\
		&+&\sum_{j=0}^{m-1}(t_m-t_j)^{\alpha-1}E_{\alpha,\alpha}(A_h(t_m-t_j)^{\alpha})P_hB(X^h_j)(W_{t_{j+1}}-W_{t_j})\nonumber\\
		&+&\sum_{j=0}^{m-1}(t_m-t_j)^{\alpha-1}E_{\alpha,\alpha}(A_h(t_m-t_j)^{\alpha})\int_{t_j}^{t_{j+1}}\int_{\mathcal{X}}P_hG(z,X^h_j)\widetilde{N}(dz,ds).
	\end{eqnarray} 
	The strong convergence result of the fully discrete scheme are formulated in the following theorem.
	\begin{thm}[Main result].
		\label{main} Let Assumptions \ref{init} - \ref{growth} are fulfilled. Let $X^h_m$ be the numerical approximation defined in \eqref{numsol}. 
		We have the following estimates depending on the regularity  parameter $\beta$ of the initial solution $X_0$ and the power of the time fractional derivative $\alpha$
		\begin{eqnarray}
			\label{mainmul1}
			\left\|X(t_m)-X^h_m\right\|_{L^2(\Omega;H)}\leq C \left(h^{\beta}+\Delta t^{\frac {\min(\alpha\beta,2-2\alpha)}2}\right). 
		\end{eqnarray}
	\end{thm}
	
	\begin{proof}
		Using triangle inequality yields
		\begin{eqnarray}
			&&\left\|X(t_m)-X^h_m\right\|_{L^2(\Omega;H)}\nonumber\\
			&\leq&\left\|X(t_m)-X^h(t_m)\right\|_{L^2(\Omega;H)}+\left\|X^h(t_m)-X^h_m\right\|_{L^2(\Omega;H)}. 
		\end{eqnarray}
		The space error is estimated in Theorem \ref{spaconv}. It remains to estimate the time error. Recall the mild solution given by \eqref{dissol1}
		\begin{eqnarray*}
			\label{dissol1*}
			&&X^h(t_m)\nonumber\\
			&=&S_{1h}(t_m)X^h_0+\sum_{j=0}^{m-1}\int_{t_j}^{t_{j+1}}(t_m-s)^{\alpha-1}S_{2h}(t_m-s)P_hF(X^h(s))ds\nonumber\\
			&+&\sum_{j=0}^{m-1}\int_{t_j}^{t_{j+1}}(t_m-s)^{\alpha-1}S_{2h}(t_m-s)P_hB(X^h(s))dW(s)\nonumber\\
			&+&\sum_{j=0}^{m-1}\int_{t_j}^{t_{j+1}}\int_{\mathcal{X}}(t_m-s)^{\alpha-1}S_{2h}(t_m-s)P_hG(z,X^h(s))\widetilde{N}(dz,ds),
		\end{eqnarray*}
		and the numerical solution $X^h_m$ given by \eqref{numsol}
		\begin{eqnarray*}
			\label{numsol*}
			X^h_m&=&S_{1h}(t_m)X^h_0+\sum_{j=0}^{m-1}\int_{t_j}^{t_{j+1}}(t_m-t_j)^{\alpha-1}S_{2h}(t_m-t_j)P_hF(X^h_j)ds\nonumber\\
			&+&\sum_{j=0}^{m-1}\int_{t_j}^{t_{j+1}}(t_m-t_j)^{\alpha-1}S_{2h}(t_m-t_j)P_hB(X^h_j)dW(s)\nonumber\\
			&+&\sum_{j=0}^{m-1}\int_{t_j}^{t_{j+1}}\int_{\mathcal{X}}(t_m-t_j)^{\alpha-1}S_{2h}(t_m-t_j)P_hG(z,X^h_j)\widetilde{N}(dz,ds).
		\end{eqnarray*}
		Subtracting these two previous equalities yields
		\begin{eqnarray*}
			&&X^h(t_m)-X^h_m\nonumber\\
			&=&\sum_{j=0}^{m-1}\int_{t_j}^{t_{j+1}} \left((t_m-s)^{\alpha-1}S_{2h}(t_m-s)P_hF(X^h(s))-(t_m-t_j)^{\alpha-1}S_{2h}(t_m-t_j)P_hF(X^h_j) \right) ds\nonumber\\
			&+&\sum_{j=0}^{m-1}\int_{t_j}^{t_{j+1}}\left( (t_m-s)^{\alpha-1}S_{2h}(t_m-s)P_hB(X^h(s))
		-(t_m-t_j)^{\alpha-1}S_{2h}(t_m-t_j)P_hB(X^h_j)\right) dW(s)\nonumber\\
			&+&\sum_{j=0}^{m-1}\int_{t_j}^{t_{j+1}}\int_{\mathcal{X}} \left((t_m-s)^{\alpha-1}S_{2h}(t_m-s)P_hG(z,X^h(s))
			-(t_m-t_j)^{\alpha-1}S_{2h}(t_m-t_j)P_hG(z,X^h_j)\right)\widetilde{N}(dz,ds)\nonumber\\
			&=:&K_1+K_2+K_3.
		\end{eqnarray*}
		Using  the triangle inequality
		\begin{eqnarray}
			\label{timerr*}
			&&\left\|X^h(t_m)-X^h_m\right\|_{L^2(\Omega;H)}\nonumber\\
			&\leq& \left\|K_1\right\|_{L^2(\Omega;H)}+\left\|K_2\right\|_{L^2(\Omega;H)}+\left\|K_3\right\|_{L^2(\Omega;H)}.
		\end{eqnarray}
		By adding and subtracting a  term, we recast $K_1$ as follows
		\begin{eqnarray}
			\label{timerr1*}
			K_1&=&\sum_{j=0}^{m-1}\int_{t_k}^{t_{k+1}}\left[(t_m-s)^{\alpha-1}S_{2h}(t_m-s)-(t_m-t_j)^{\alpha-1}S_{2h}(t_m-t_j)\right]P_hF(X^h(s))ds\nonumber\\
			&+&\sum_{j=0}^{m-1}\int_{t_k}^{t_{k+1}}(t_m-t_j)^{\alpha-1}S_{2h}(t_m-t_j)P_h\left[F(X^h(s))-F(X^h(t_k))\right]ds\nonumber\\
			&+&\sum_{j=0}^{m-1}\int_{t_k}^{t_{k+1}}(t_m-t_j)^{\alpha-1}S_{2h}(t_m-t_j)P_h\left[F(X^h(t_k))-F(X^h_k)\right]ds\nonumber\\
			&=:&K_{11}+K_{12}+K_{13}.
		\end{eqnarray}
		Using triangle inequality, the discrete version of \eqref{semigroup_prp1*}, Assumption \ref{nonlin} with, boundedness of $P_h$ and the discrete version of \eqref{boun}, Cauchy-Schwartz inequality and the variable change $k=m-j$ leads
		
		\begin{eqnarray}
			\label{timerr11*}
			&&\left\|K_{11}\right\|_{L^2(\Omega;H)}\nonumber\\
			&\leq&\sum_{j=0}^{m-1}\int_{t_j}^{t_{j+1}}\left\|(t_m-s)^{\alpha-1}S_{2h}(t_m-s)-(t_m-t_j)^{\alpha-1}S_{2h}(t_m-t_j)\right\|_{L(H)}\nonumber\\
			&&\hspace{3cm}\left\|P_hF(X^h(s))\right\|_{L^2(\Omega;H)}ds\nonumber\\
			&\leq&C\left(\sum_{j=0}^{m-1}\int_{t_j}^{t_{j+1}}(s-t_j)^{1-\alpha}(t_m-t_j)^{\alpha-1}(t_m-s)^{\alpha-1}ds\right)\left(1+\mathbb{E}\left[\sup_{0\leq s\leq T}\left\|X^h(s)\right\|^2\right]\right)^{\frac 12}\nonumber\\
			&\leq& C\Delta t^{1-\alpha}\left(\sum_{j=0}^{m-1}(t_m-t_j)^{\alpha-1}\int_{t_j}^{t_{j+1}}(t_m-s)^{\alpha-1}ds\right)\nonumber\\
			&\leq& C\Delta t^{1-\alpha}\left(\sum_{j=0}^{m-1}(t_m-t_j)^{\alpha-1}\right)^{\frac 12}\left(\sum_{j=0}^{m-1}\left(\int_{t_j}^{t_{j+1}}(t_m-s)^{2\alpha-2}ds\right)^2\right)^{\frac 12}\nonumber\\
			&\leq& C\Delta t^{1-\alpha}\left(\sum_{k=1}^{m}t_k^{-1+2\alpha-1}\right)^{\frac 12}\left(\Delta t\int_{t_j}^{t_{j+1}}(t_m-s)^{2\alpha-2}ds\right)^{\frac 12}\nonumber\\
			&\leq& C\Delta t^{1-\alpha}\left(\sum_{k=1}^{m}t_k^{-1+2\alpha-1}\Delta t\right)^{\frac 12}t_m^{\alpha-\frac 12}.
		\end{eqnarray}
		Let us recall the following estimate, for $\epsilon>0$ small enough
		\begin{eqnarray}
			\label{est}
			\sum_{k=1}^{m}t_k^{-1+\epsilon}\Delta t\leq C.
		\end{eqnarray}
		Inserting \eqref{est} in \eqref{timerr11*} with $\epsilon=2\alpha-1$ yields
		\begin{eqnarray}
			\label{timerr11}
			\left\|K_{11}\right\|_{L^2(\Omega;H)}\leq C\Delta t^{1-\alpha}.
		\end{eqnarray}
		For the second estimate $\left\|K_{12}\right\|_{L^2(\Omega;H)}$, applying triangle inequality, the boundedness of $S_{2h}(t)$ and $P_h$, Assumption \ref{nonlin}, \eqref{tim_reg1}, the variable change $k=m-j$ and \eqref{est} with $\epsilon=\alpha$ yields
		
		\begin{eqnarray}
			\label{timerr12}
			&&\left\|K_{12}\right\|_{L^2(\Omega;H)}\nonumber\\
			&\leq& \sum_{j=0}^{m-1}\int_{t_j}^{t_{j+1}}\left\|(t_m-t_j)^{\alpha-1}S_{2h}(t_m-t_j)P_h\left(F(X^h(s))-F(X^h(t_j))\right)\right\|_{L^2(\Omega;H)}ds\nonumber\\
			&\leq& \sum_{j=0}^{m-1}\int_{t_k}^{t_{k+1}}(t_m-t_j)^{\alpha-1}\left\|S_{2h}(t_m-t_j)P_h\right\|_{L(H)}\left\|F(X^h(s))-F(X^h(t_j))\right\|_{L^2(\Omega;H)}ds\nonumber\\
			&\leq& \sum_{j=0}^{m-1}\int_{t_k}^{t_{k+1}}(t_m-t_j)^{\alpha-1}(s-t_j)^{\frac{\min(\alpha\beta,2-2\alpha)}2}ds\nonumber\\
			&\leq& C \Delta t^{1+\frac{\min(\alpha\beta,2-2\alpha)}2}\sum_{j=0}^{m-1}(t_m-t_j)^{\alpha-1}\nonumber\\
			&\leq& C \Delta t^{\frac{\min(\alpha\beta,2-2\alpha)}2}\left(\sum_{j=0}^{m-1}(t_m-t_j)^{-1+\alpha}\Delta t\right) \leq C \Delta t^{\frac{\min(\alpha\beta,2-2\alpha)}2},
		\end{eqnarray}
		and using also triangle inequality boundedness of $S_{2h}(t)$ and $P_h$, Assumption \ref{nonlin}, we estimate $\left\|K_{13}\right\|_{L^2(\Omega;H)}$ as follows
		\begin{eqnarray}
			\label{timerr13}
			&&\left\|K_{13}\right\|_{L^2(\Omega;H)}\nonumber\\
			&\leq& \sum_{j=0}^{m-1}\int_{t_j}^{t_{j+1}}\left\|(t_m-t_j)^{\alpha-1}S_{2h}(t_m-t_j)P_h\left(F(X^h(t_j))-F(X^h_j)\right)\right\|_{L^2(\Omega;H)}ds\nonumber\\
			&\leq& \sum_{j=0}^{m-1}\int_{t_j}^{t_{j+1}}(t_m-t_j)^{\alpha-1}\left\|S_{2h}(t_m-t_j)P_h\right\|_{L(H)}\left\|F(X^h(t_j))-F(X^h_j)\right\|_{L^2(\Omega;H)}ds\nonumber\\
			&\leq& C \Delta t \sum_{j=0}^{m-1}(t_m-t_j)^{\alpha-1}\left\|X^h(t_j)-X^h_j\right\|_{L^2(\Omega;H)}\nonumber\\
			&\leq& C \Delta t^{\alpha} \sum_{j=0}^{m-1}\left\|X^h(t_j)-X^h_j\right\|_{L^2(\Omega;H)}
		\end{eqnarray}
		Adding \eqref{timerr11} - \eqref{timerr13}, it holds that
		\begin{eqnarray}
			\label{timerr1}
			\left\|K_1\right\|_{L^2(\Omega;H)}\leq C \Delta t^{\frac{\min(\alpha\beta,2-2\alpha)}2}+C \Delta t^{\alpha} \sum_{k=0}^{m-1}\left\|X^h(t_k)-X^h_k\right\|_{L^2(\Omega;H)}.
		\end{eqnarray}
		We will not give details of the estimate of $K_2$ as it is similar to that of $K_3$. Let us now estimate the norm of $K_3$. By adding and subtracting the same term, we rewrite it in three terms as follows
		\begin{eqnarray}
			\label{timerr3*}
			K_3&=&\sum_{j=0}^{m-1}\int_{t_j}^{t_{k+1}}\int_{\mathcal{X}}\left[(t_m-s)^{\alpha-1}S_{2h}(t_m-s)P_hG(z,X^h(s))-(t_m-t_j)^{\alpha-1}S_{2h}(t_m-t_j)\right]\nonumber\\
			&&\hspace{3cm}P_hG(z,X^h(s))\widetilde{N}(dz,ds)\nonumber\\
			&+&\sum_{j=0}^{m-1}\int_{t_j}^{t_{j+1}}\int_{\mathcal{X}}(t_m-t_j)^{\alpha-1}S_{2h}(t_m-t_j)P_h\left[G(z,X^h(s))-G(z,X^h(t_j))\right]\widetilde{N}(dz,ds)\nonumber\\
			&+&\sum_{j=0}^{m-1}\int_{t_j}^{t_{j+1}}\int_{\mathcal{X}}(t_m-t_j)^{\alpha-1}S_{2h}(t_m-t_j)P_h\left[G(z,X^h(t_k))-G(z,X^h_j)\right]\widetilde{N}(dz,ds)\nonumber\\
			&=:&K_{31}+K_{32}+K_{33}.
		\end{eqnarray}
		Applying again the It\^o isometry property \eqref{jumpint}, the fact that the variation of the compensated Poisson measure are independent, the discrete version of \eqref{semigroup_prp1*} with $t_2=t_m-t_j$ and $t_1=t_m-s$, Assumption \ref{growth} with $\tau=0$, boundedness of $P_h$, the discrete version of \eqref{boun} and Cauchy-Schwartz inequality leads
		\begin{eqnarray}
			\label{timerr31*}
			&&\left\|K_{31}\right\|^2_{L^2(\Omega;H)}\nonumber\\
			&=&\left\|\sum_{j=0}^{m-1}\int_{t_j}^{t_{j+1}}\int_{\mathcal{X}}\left[(t_m-s)^{\alpha-1}S_{2h}(t_m-s)-(t_m-t_j)^{\alpha-1}S_{2h}(t_m-t_j)\right]\right.\nonumber\\
			&&\hspace{1cm}\left.P_hG(z,X^h(s))\widetilde{N}(dz,ds)\right\|^2_{L^2(\Omega;H)}\nonumber\\
			&=&\sum_{j=0}^{m-1}\mathbb{E}\left[\int_{t_j}^{t_{j+1}}\int_{\mathcal{X}}\left\|(t_m-s)^{\alpha-1}S_{2h}(t_m-s)-(t_m-t_j)^{\alpha-1}S_{2h}(t_m-t_j)\right\|^2_{L(H)}\right.\nonumber\\
			&&\hspace{1cm}\left.\left\|P_hG(z,X^h(s))\right\|^2\upsilon(dz)ds\right]\nonumber\\
			&\leq&C\sum_{j=0}^{m-1}\int_{t_j}^{t_{j+1}}(s-t_j)^{2-2\alpha}(t_m-t_j)^{2\alpha-2}(t_m-s)^{2\alpha-2}\nonumber\\
			&&\hspace{2cm}\mathbb{E}\left[\int_{\mathcal{X}}\left\|P_hG(z,X^h(s))\right\|^2\upsilon(dz)\right]ds\nonumber\\
			&\leq& C\Delta t^{2-2\alpha}\left(\sum_{j=0}^{m-1}(t_m-t_j)^{2\alpha-2}\int_{t_j}^{t_{j+1}}(t_m-s)^{2\alpha-2}\left(1+\mathbb{E}\left[\left\|X^h(s)\right\|^2\right]\right)ds\right)\nonumber\\
			&\leq& C\Delta t^{2-2\alpha}\left(\sum_{j=0}^{m-1}(t_m-t_j)^{2\alpha-2}\int_{t_j}^{t_{j+1}}(t_m-s)^{2\alpha-2}ds\right)\left(1+\mathbb{E}\left[\sup_{0\leq s\leq T}\left\|X^h(s)\right\|^2\right]\right)\nonumber\\
			&\leq& C\Delta t^{2-2\alpha}\left(\sum_{j=0}^{m-1}(t_m-t_j)^{4\alpha-4}\right)^{\frac12}\left(\sum_{j=0}^{m-1}\left(\int_{t_j}^{t_{j+1}}(t_m-s)^{2\alpha-2}ds\right)^2\right)^{\frac12},
		\end{eqnarray}
		using additionally the variable change $k=m-j$ and \eqref{est} with $\epsilon=4\alpha-3$ yields
		\begin{eqnarray}
			\label{timerr31}
			\left\|K_{31}\right\|^2_{L^2(\Omega;H)}&\leq& C\Delta t^{2-2\alpha}\left(\sum_{k=1}^{m}t_k^{-1+4\alpha-3}\right)^{\frac12}\left(\Delta t\sum_{j=0}^{m-1}\int_{t_j}^{t_{j+1}}(t_m-s)^{4\alpha-4}ds\right)^{\frac 12}\nonumber\\
			&\leq& C\Delta t^{2-2\alpha}\left(\sum_{k=1}^{m}t_k^{-1+4\alpha-3}\Delta t\right)^{\frac12}\left(\int_0^{t_m}(t_m-s)^{4\alpha-4}ds\right)^{\frac 12}\nonumber\\
			&\leq& C\Delta t^{2-2\alpha}\left(\sum_{k=1}^{m}t_k^{-1+4\alpha-3}\Delta t\right)^{\frac12}t_m^{\frac {4\alpha-3}2}\leq C\Delta t^{2-2\alpha}.
		\end{eqnarray}
		To estimate the second term $\left\|K_{32}\right\|^2_{L^2(\Omega;H)}$, applying It\^o isometry \eqref{jumpint}, using boundedness of $S_{2h}(t)$, $P_h$ and Assumption \ref{lips}, Theorem \ref{tim_reg}, the variable change $k=m-j$ and \eqref{est} with $\epsilon=2\alpha-1$ yields
		\begin{eqnarray}
			\label{timerr32}
			&&\left\|K_{32}\right\|^2_{L^2(\Omega;H)}\nonumber\\
			&=&\left\|\sum_{j=0}^{m-1}\int_{t_k}^{t_{k+1}}\int_{\mathcal{X}}(t_m-t_j)^{\alpha-1}S_{2h}(t_m-t_j)P_h\left[G(z,X^h(s))-G(z,X^h(t_k))\right]\right.\nonumber\\
			&&\hspace{3cm}\left.\widetilde{N}(dz,ds)\right\|^2_{L^2(\Omega;H)}\nonumber\\
			&=&\sum_{j=0}^{m-1}\mathbb{E}\left[\int_{t_j}^{t_{j+1}}\int_{\mathcal{X}}\left\|(t_m-t_j)^{\alpha-1}S_{2h}(t_m-t_j)P_h\left[G(z,X^h(s))-G(z,X^h(t_j))\right]\right\|^2\upsilon(dz)ds\right]\nonumber\\
			&\leq&\sum_{j=0}^{m-1}\mathbb{E}\left[\int_{t_j}^{t_{j+1}}(t_m-t_j)^{2\alpha-2}\left\|S_{2h}(t_m-t_j)P_h\right\|^2_{L(H)}\int_{\mathcal{X}}\left\|G(z,X^h(s))-G(z,X^h(t_j))\right\|^2\upsilon(dz)ds\right]\nonumber\\
			&\leq&C\sum_{j=0}^{m-1}\int_{t_j}^{t_{j+1}}(t_m-t_j)^{2\alpha-2}\left\|X^h(s))-X^h(t_j)\right\|^2_{L^2(\Omega;H)}ds\nonumber\\
			&\leq&C\sum_{j=0}^{m-1}\int_{t_j}^{t_{j+1}}(t_m-t_j)^{2\alpha-2}(s-t_j)^{\min(\alpha\beta,2-2\alpha)}ds\nonumber\\
			&\leq&C\Delta t^{1+\min(\alpha\beta,2-2\alpha)}\left(\sum_{j=0}^{m-1}(t_m-t_j)^{2\alpha-2}\right)\nonumber\\
			&\leq&C\Delta t^{\min(\alpha\beta,2-2\alpha)}\left(\sum_{k=1}^{m}t_k^{-1+2\alpha-1}\Delta t\right)\leq C\Delta t^{\min(\alpha\beta,2-2\alpha)}.
		\end{eqnarray}
		To estimate the third term $\left\|K_{33}\right\|^2_{L^2(\Omega;H)}$, applying It\^o isometry \eqref{jumpint}, using boundedness of $S_{2h}(t)$, $P_h$ and Assumption \ref{lips} yields
		\begin{eqnarray}
			\label{timerr33}
			&&\left\|K_{33}\right\|^2_{L^2(\Omega;H)}\nonumber\\
			&=&\left\|\sum_{j=0}^{m-1}\int_{t_j}^{t_{j+1}}\int_{\mathcal{X}}(t_m-t_j)^{\alpha-1}S_{2h}(t_m-t_j)P_h\left[G(z,X^h(t_j))-G(z,X^h_j)\right]\widetilde{N}(dz,ds)\right\|^2_{L^2(\Omega;H)}\nonumber\\
			&=&\sum_{j=0}^{m-1}\mathbb{E}\left[\int_{t_j}^{t_{j+1}}\int_{\mathcal{X}}\left\|(t_m-t_j)^{\alpha-1}S_{2h}(t_m-t_j)P_h\left[G(z,X^h(t_j))-G(z,X^h_j)\right]\right\|^2\upsilon(dz)ds\right]\nonumber\\
			&\leq&\sum_{j=0}^{m-1}\mathbb{E}\left[\int_{t_j}^{t_{j+1}}(t_m-t_j)^{2\alpha-2}\left\|S_{2h}(t_m-t_j)P_h\right\|^2_{L(H)}\int_{\mathcal{X}}\left\|G(z,X^h(t_j))-G(z,X^h_j)\right\|^2\upsilon(dz)ds\right]\nonumber\\
			&\leq&C\sum_{j=0}^{m-1}(t_m-t_j)^{2\alpha-2}\int_{t_j}^{t_{j+1}}\left\|X^h(t_j)-X^h_j\right\|^2_{L^2(\Omega;H)}ds\nonumber\\
			&\leq&C\Delta t\sum_{j=0}^{m-1}(t_m-t_j)^{2\alpha-2}\left\|X^h(t_j)-X^h_j\right\|^2_{L^2(\Omega;H)}\nonumber\\
			&\leq&C\Delta t^{2\alpha-1}\sum_{j=0}^{m-1}\left\|X^h(t_j)-X^h_j\right\|^2_{L^2(\Omega;H)}.
		\end{eqnarray}
		Substituting \eqref{timerr31} - \eqref{timerr33} in \eqref{timerr3*} leads
		\begin{eqnarray}
			\label{timerr3}
			\left\|K_3\right\|^2_{L^2(\Omega;H)}\leq C \Delta t^{\min(\alpha\beta,2-2\alpha)}+C\Delta t^{2\alpha-1}\sum_{k=0}^{m-1}\left\|X^h(t_k)-X^h_k\right\|^2_{L^2(\Omega;H)}.
		\end{eqnarray}
		Using the similar procedure as for $K_3$, we obtain the following estimate of the norm of $K_2$
		\begin{eqnarray}
			\label{timerr2}
			\left\|K_2\right\|^2_{L^2(\Omega;H)}\leq C \Delta t^{\min(\alpha\beta,2-2\alpha)}+C\Delta t^{2\alpha-1}\sum_{k=0}^{m-1}\left\|X^h(t_k)-X^h_k\right\|^2_{L^2(\Omega;H)}.
		\end{eqnarray}
		Substituting \eqref{timerr1}, \eqref{timerr3} and \eqref{timerr2} in \eqref{timerr*} yields
		\begin{eqnarray}
			\label{timerr}
			&&\left\|X(t_m)-X^h_m\right\|^2_{L^2(\Omega;H)}\nonumber\\
			&\leq& C \Delta t^{\min(\alpha\beta,2-2\alpha)}+C\Delta t^{2\alpha-1}\sum_{k=0}^{m-1}\left\|X^h(t_k)-X^h_k\right\|^2_{L^2(\Omega;H)}.
		\end{eqnarray} 
		Applying discrete Gronwall's lemma to \eqref{timerr} and taking the square root leads
		\begin{eqnarray}
			\label{tim_conv}
			\left\|X(t_m)-X^h_m\right\|_{L^2(\Omega;H)}\leq C \Delta t^{\frac {\min(\alpha\beta,2-2\alpha)}2}.
		\end{eqnarray}
		Combining \eqref{tim_conv} and Theorem \ref{spaconv} completes the proof of Theorem \ref{main}.
	\end{proof}


\begin{thebibliography}{00}
		\bibitem{Alm}
		Al-Maskari, M., Karaa, S.: 
		\newblock{Numerical approximation of semilinear subdiffusion equations with nonsmooth initial data}. 
		\newblock{SIAM J. Numer. Anal. \textbf{57}(3),  1524--1544} (2019)
		
		
		
		
		
		\bibitem{Tankov}
		Cont, R., Tankov, P.: 
		\newblock{Financial modelling with jump process.}
		\newblock{ In: Financial Mathematics series. CRC Press. Boca, FL (2000)}
		
		
		
		
		\bibitem{Elz} 
		Elzaki, T. M., Daoud, Y., Biazar, J.:
		\newblock{Decomposition Method for Fractional Partial Differential Equations Using Modified Integral Transform}. 
		\newblock{World Applied Sciences Journal. \textbf{37}(1), 18--24 (2019)}
		
		\bibitem{For}
		Ford, N. J., Xiao J., Yan Y.:  
		\newblock{A finite element method for time fractional partial differential equation}. 
		\newblock{Fractional Calculus and Aplied analysis (2011). https://doi.org/10.2478/s13540-011-0028-2}
		
		\bibitem{Fuj} 
		Fujita, F., Suzuki, T.:
		\newblock{ Evolution problems (Part 1). Handbook of Numerical Analysis (P. G. Ciarlet and J.L. Lions eds)}.
		\newblock{\textbf{2} Amsterdam, The Netherlands: North-Holland ,789--928 (1991)}
		
		\bibitem{Gao} 
		Gao, G. H., Sun b, Z. Z., Zhang, H. W.:
		\newblock{A new fractional numerical differentiation formula to approximate the Caputo fractional derivative and its applications}.
		\newblock{Journal of Computational Physics \textbf{259}, 33--50 (2014)}
		
		\bibitem{ExpF}
		Garrappa R.: 
		\newblock{ A family of Adams exponential integrators for fractional linear systems}.
		\newblock{Computers \& Mathematics with Applications  \textbf{66}(5), 717--727 (2013)}
		
		\bibitem{ML2}
		Garrappa, R., Popolizio, M.:
		\newblock{Computing the Matrix Mittag-Leffler Function with Applications to Fractional Calculus}
		\newblock{ J Sci Comput} \textbf{77}, 129--153 (2018)
		
		\bibitem{Gun}
		Gunzburger, M., Li, B., Wang, J.:
		\newblock{Sharp convergence rates of time discretization for srochastic time-fractional PDEs subject to additive space-time white noise}.
		\newblock{arXiv:1704.02912v2 [math.NA] 8 Aug 2018.} 
		
		\bibitem{Hau}
		Haubold, H. J., Mathai, A. M., Saxena, R. K.:  \newblock{Mittag-Leffler functions and their applications}. 
		\newblock{J Appl Math (2011) (Article ID 298628)}
		
		\bibitem{Henry}
		Henry, D.:
		\newblock{Geometric Theory of semilinear parabolic Equations.}
		\newblock{Lecture notes in Mathematics. \textbf{840}. Berlin : Springer, (1981)}
		
		
		\bibitem{Jia}
		Jianga, Y., Ma, J.:
		\newblock{High-order finite element methods for time-fractional partial
			differential equations}.
		\newblock{Journal of Computational and Applied Mathematics \textbf{235}, 3285--3290 (2011)}
		
		
		
		
		
		\bibitem{Krub}
		Kruse, R.: 
		\newblock{Strong and weak approximation of semilinear stochastic evolution equations}.
		\newblock{Springer, New-York (2014)}
		
		
		
		\bibitem{Lia}
		Li, X., Yang, X., Zhang, Y.: 
		\newblock{Error Estimates of Mixed Finite Element Methods for Time-Fractional Navier-Stokes Equations}. 
		\newblock{J Sci Comput  \textbf{70}, 500--515 (2017)}
		
		\bibitem{Lin}
		Lin, Y., Xu, C.:  
		\newblock{Finite difference/spectral approximations for
			the time-fractional diffusion equation}.
		\newblock{Journal of Computational Physics \textbf{225},  1533--1552 (2007)}
		
		\bibitem{Liu}
		Liu, Y., Li, H., Gao, W., He, S., Fang, Z.: 
		\newblock{A New Mixed Element Method for a Class of Time-Fractional Partial Differential Equations}.
		\newblock{The Scientific World Journal (2014).  http://dx.doi.org/10.1155/2014/141467}
		
		\bibitem{Lor}
		Lord, G. J., Tambue, A.:
		\newblock{Stochastic exponential integrators for the finite element discretization of SPDEs for multiplicative \& additive noise,}
		\newblock{IMA J Numer. Anal. \textbf{2},  515--543 (2013)}
		
		\bibitem{Mai} 
		Mainardi, F.: 
		\newblock{The fundamental solutions for the fractional diffusion-wave equation.} 
		\newblock{Appl. Math. Lett. \textbf{9}, 23--28 (1996).}
		
		
		\bibitem{Man}
		Mandreka, V., R\"{u}diger, B.:
		\newblock{Stochastic integration in Banach space, in: Probability Theory and Stochastic Modelling, in: Theory and applications.}
		\newblock{Springer, Cham (2015)}
		
		
		
		\bibitem{ML3}
		Moret, I., Novati, P.:
		\newblock{On the Convergence of Krylov Subspace Methods for Matrix Mittag-Leffler Functions.}
		\newblock{SIAM Journal on  Numerical Analysis} \textbf{49}(5), 2144--216 (2011)
		
		
		
		
		\bibitem{manto}
		Mukam, J. D., Tambue, A.:
		\newblock{Optimal strong convergence rates of numerical methods for semilinear parabolic SPDE driven by Gaussian noise and Poisson random measure.}
		\newblock{ Computers \& Mathematics with Applications} \textbf{77}(10), 2786--2803 (2019)
		
		\bibitem{Nou}
		Noupelah, A. J., Tambue, A.:
		\newblock{Optimal strong convergence rates of some Euler-type timestepping schemes for the finite element discretization SPDEs driven by additive fractional Brownian motion and Poisson random measure.}
		\newblock{Preprint,  arXiv:1912.12751 (2019)}
		
		\bibitem{Osm} 
		Osman, S. A., Langlands, T. A. M.:
		\emph{An implicit Keller Box numerical scheme for the solution of fractional subdiffusion equations}. 
		\newblock{Applied Mathematics and Computation \textbf{348}, 609--626 (2019). https://doi.org/10.1016/j.amc.2018.12.015}
		
		\bibitem{Pazy}
		Pazy, A.: 
		\newblock {Semigroups of Linear Operators and Applications to Partial Differential Equations.}
		\newblock{Applied Mathematical Sciences. Springer-Verlag, New York, v 44, (1983)}
		
		\bibitem{Platen1}
		Platen, E., Bruti-Liberati, N.:
		\newblock{Numerical solution of stochastic differential equations with jumps in finance.}
		\newblock{ in: Stochastic Modelling and Applied Probability.} Volume 64, Springer-Verlag, Berlin (2010)
		
		\bibitem{ML1}
		Popolizio, M.: 
		\newblock{ On the Matrix Mittag-Leffler Function: Theoretical Properties and Numerical Computation.}
		\newblock{ Mathematics, (2019). https://doi.org/10.3390/math7121140}
		
		\bibitem{Pra}
		Prato D., Zabczyk G. J.:
		\newblock{Stochastic Equations in Infinite Dimensions.}
		\newblock{Cambridge University Press. Cambridge. United Kingdom, vol. 152, (2014)}
		
		\bibitem{Prevot}
		Pr\'{e}v\^{o}t, C., R\"{o}ckner, M.:
		\newblock{ A Concise Course on Stochastic Partial Differential Equations.}
		\newblock{Lecture Notes in Mathematics,  Springer: Berlin, vol. 1905, (2007)}
		
		\bibitem{Pri}
		Priya, G. S., Prakash, P., Nieto, J. J., Kayar, Z.: 
		\newblock{Higher Order Numerical Scheme for the Fractional Heat Equation with Dirichlet and Neumann Boundary Conditions}. 
		\newblock{Numerical Heat Transfer, Part B: Fundamentals, \textbf{63}(6), 540--559 (2013)}
		
		
		
		
		\bibitem{ATthesis}
		Tambue, A.:  
		\newblock{ Efficient Numerical Schemes for Porous Media Flow.}
		\newblock { PhD Thesis, Department of Mathematics, Heriot--Watt University} (2010)
		
		
		
		
		
		
		
		
		
		
		\bibitem{Yan} 
		Yang, Y., Zeng, F.: 
		\newblock{Numerical analysis of linear and nonlinear time-fractional subdiffusion equations}.
		\newblock{arXiv:1901.06814v1 [math.NA] 21 Jan 2019.}
		
		\bibitem{Ye}
		Ye, H., Gao, J., Ding,  Y.:
		\newblock{A generalized Gronwall inequality and its application to a fractional differential equation}.  
		\newblock{J. Math. Anal. Appl. \textbf{328}, 1075--1081 (2007)}
		
		\bibitem{Zha}
		Zhang, L., Ding, Y., Hao, K., Hu, L.:  
		\newblock{Moment stability of fractional stochastic evolution equations with Poisson jumps}.
		\newblock{ International Journal of Systems Science. \textbf{45}(7), 1539--1547 (2014)}
		
		\bibitem{Zoua}
		Zou, G. A., Wang, B., Zhou, Y.:
		\newblock{Existence and regularity of mild solution to
			fractional stochastic evolution equations}.
		\newblock{Math. Model. Nat. Phenom. 13 (2018). https://doi.org/10.1051/mmnp/2018004} 
		
		\bibitem{Zoub}
		Zou, G. A.: 
		\newblock{Galerkin finite element method for time-fractional stochastic diffusion equations},
		\newblock{Comp. Appl. Math. (2018). https://doi.org/10.1007/s40314-018-0609-3}
		
		
		%
		%
		
	\end{thebibliography}
\end{document}